\documentclass[11pt]{amsart}

\title[Stable maps and hyperbolic links]
{Stable maps and hyperbolic links}

\author{Ryoga Furutani}
\address{
Department of Mathematics \newline
\indent Hiroshima University, 1-3-1 Kagamiyama, Higashi-Hiroshima, 739-8526, Japan}
\email{ryoga.furutani0409@gmail.com}

\author{Yuya Koda}
\address{
Department of Mathematics \newline
\indent Hiroshima University, 1-3-1 Kagamiyama, Higashi-Hiroshima, 739-8526, Japan}
\email{ykoda@hiroshima-u.ac.jp}

\usepackage{amsthm}
\usepackage{mathrsfs}
\usepackage{latexsym}
\usepackage[dvipdfmx]{graphicx}
\usepackage[dvips]{psfrag}
\usepackage[dvips]{color}
\usepackage{xypic}
\usepackage[abs]{overpic}
\usepackage{caption}
\usepackage[all]{xy}
\usepackage[normalem]{ulem}

\theoremstyle{plain}
\newtheorem*{theorem*}{Theorem}
\newtheorem*{lemma*} {Lemma}
\newtheorem*{corollary*} {Corollary}
\newtheorem*{proposition*}{Proposition}
\newtheorem*{conjecture*}{Conjecture}
\newtheorem{theorem}{Theorem}[section]

\newtheorem{corollary}[theorem]{Corollary}
\newtheorem{proposition}[theorem]{Proposition}

\newtheorem{claim}{Claim}

\theoremstyle{remark}

\newtheorem*{definition}{Definition}
\newtheorem*{claim*}{Claim}

\newtheorem{example}{Example}

\theoremstyle{definition}

\newtheoremstyle{citing}
  {}
  {}
  {\itshape}
  {}
  {\bfseries}
  {.}
  {.5em}
  {\thmnote{#3}}

\theoremstyle{citing}

\newcommand{\Integer}{\mathbb{Z}}
\newcommand{\Real}{\mathbb{R}}
\newcommand{\Complex}{\mathbb{C}}

\newcommand{\Nbd}{\operatorname{Nbd}}

\newcommand{\Int}{\operatorname{Int}}

\newcommand{\gl}{\operatorname{gl}}



\begin{document}

\maketitle

\begin{abstract}
A stable map of a closed orientable $3$-manifold into the real plane is called 
a stable map of a link in the manifold if the link is contained in the set of definite fold points.   
We give a complete characterization of the hyperbolic links in the $3$-sphere that admit stable maps 
into the real plane with exactly one (connected component of a) fiber having two singular points. 
\end{abstract}

\vspace{1em}

\begin{small}
\hspace{2em}  \textbf{2020 Mathematics Subject Classification}: 
57R45; 57K10, 57K32, 57R05


\hspace{2em} 
\end{small}

\section*{Introduction}

Let $f: M \to \Real^2$ be a stable map of a closed orientable $3$-manifold $M$ into the real plane $\Real^2$. 
It is well-known that the singular points of $f$ are classified into three types: 
definite fold, indefinite fold and cusp. 
In \cite{Lev65} Levine showed that the cusp points can always be eliminated by 
a homotopical deformation. 
This implies that every $M$ admits a stable map into $\Real^2$ without
cusp points. 
In the following, we only consider stable maps without cusp points. 
The properties of the other types of singular points 
reflect the global topology of the source manifold $M$ as follows. 
In \cite{BR74} Burlet-de Rham showed that if $M$ admits a stable map into $\Real^2$ 
with only definite fold points, then $M$ is a connected sum $\#_g (S^2 \times S^1)$ 
for some non-negative integer $g$, where the empty connected sum $\#_0 (S^2 \times S^1)$ 
implies $S^3$. 
In general, each connected component of a fiber of a stable map $f: M \to \Real^2$ 
contains at most two singular points, and a 
connected component of a fiber of $f$ containing exactly two singular points, both of which 
are indefinite fold, is one of types $\mathrm{II}^2$ and $\mathrm{II}^3$ singular fibers 
shown in Figure \ref{figure:vertexfiber}. 
For any stable map of $M$ into $\Real^2$, there exists at most finitely many singular fibers of 
these types. 
\begin{figure}[htbp]
\centering\includegraphics[width=5cm]{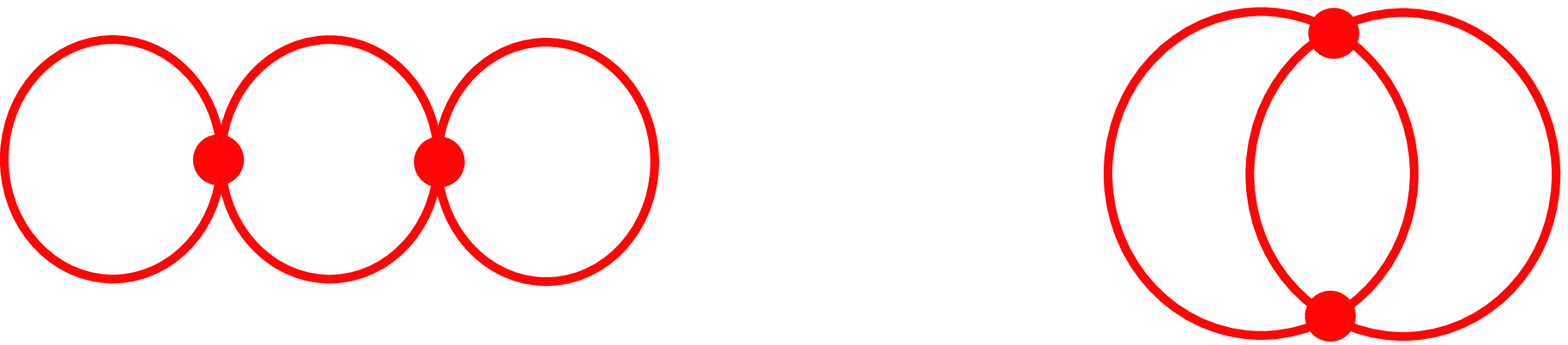}
\begin{picture}(400,0)(0,0)
\put(140,0){Type $\mathrm{II}^2$}
\put(230,0){Type $\mathrm{II}^3$}
\end{picture}
\caption{Types $\mathrm{II}^2$ and $\mathrm{II}^3$ singular fibers.}
\label{figure:vertexfiber}
\end{figure}
In \cite{Sae96} Saeki generalized the result of \cite{BR74} showing that 
if $M$ admits a stable map into $\Real^2$ with no 
$\mathrm{II}^2$ or $\mathrm{II}^3$ singular fibers, 
then $M$ is a graph manifold, and vice versa. 
Recall that a $3$-manifold is a graph manifold if and only if its Gromov norm is zero.  
Later, Costantino-Thurston \cite{CT08} and Gromov \cite{Gro09} gave, independently, 
a linear lower bound 
of the number of types $\mathrm{II}^2$ and $\mathrm{II}^3$ singular fibers 
of stable maps of $M$ into $\Real^2$ 
in terms of Gromov norm.

In the present paper, we are interested in stable maps of links. 
Again, let $f: M \to \Real^2$ be a stable map. 
We note that the set $S(f)$ of singular points of $f$ forms a link in $M$, each of whose components consists only of 
exactly one type of singular points. 
When a link $L$ in $M$ is contained in the set of definite fold points of $f$, we say that  
$f$ is a stable map of $(M, L)$.   
In \cite{Sae96} Saeki showed that if $(M, L)$ admits a stable map 
with no singular fibers of type $\mathrm{II}^2$ or $\mathrm{II}^3$, 
then $L$ is a graph link, that is, the exterior of $L$ is a graph manifold, and vice versa.
This implies that any stable map of a hyperbolic link $L$ in $M$ 
has at least one singular fiber of type $\mathrm{II}^2$ or $\mathrm{II}^3$. 
Recently, Ishikawa-Koda \cite{IK17} gave a complete characterization of 
the hyperbolic links in $S^3$ admitting stable maps into $\Real^2$ 
with a single singular fiber type $\mathrm{II}^{2}$ and 
no one of type $\mathrm{II}^{3}$, see Theorem~\ref{thm:classification by Ishikawa-Koda}. 
For other work on the study of links through stable maps, see e.g. \cite{Sae94,KS12}. 

The following is the main theorem of this paper. 
\begin{theorem} 
\label{thm:main theorem}
Let $L$ be a hyperbolic link in $S^3$. 
Then there exists a stable map $f:(S^3,L) \to \Real^2$ with 
$\mathrm{II}^{2}(f) = \emptyset$ and  $|\mathrm{II}^{3}(f)| = 1$ if and only if 
the exterior of $L$ is diffeomorphic to
a $3$-manifold obtained by Dehn filling the exterior of one of the 
four links $L'_{1}, L'_2, L'_3$ and $L'_4$ in $S^3$ along some of 
$($possibly none of$)$ boundary tori, 
where $L'_{1}, L'_2, L'_3, L'_4$ are 
depicted in Figure $\ref{figure:Furutanilink}$. 
\end{theorem}
\begin{figure}[htbp]
\centering\includegraphics[width=14cm]{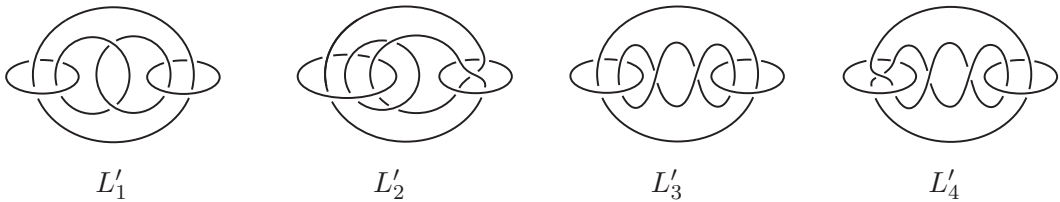}
\begin{picture}(400,0)(0,0)
\put(35,0){$L'_{1}$}
\put(140,0){$L'_{2}$}
\put(245,0){$L'_{3}$}
\put(350,0){$L'_{4}$}
\end{picture}
\caption{The links $L'_{1}, L'_2, L'_3$ and $L'_4$.}
\label{figure:Furutanilink}
\end{figure}

We note that the link $L'_1$ above is the minimally twisted 5-chain link. 
All of the above links $L'_{1}, L'_2, L'_3, L'_4$ are invariant by the rotation of $\pi$ around a horizontal axis. 
Theorem~\ref{thm:main theorem}, together with the above result of Ishikawa-Koda \cite{IK17}, completes the 
characterization of the hyperbolic links in $S^3$ that admit stable maps 
into $\Real^2$ with exactly one (connected component of a) fiber having two singular points. 
The proof is based on the connection between the Stein factorizations of stable maps and 
branched shadows developed in Costantino-Thurston \cite{CT08} and Ishikwa-Koda \cite{IK17}. 
Our proof is constructive, thus, as we will see in 
Corollary~\ref{cor:configuration of the unique singular fiber of type II3} 
we can actually describe the configuration of the fibers of the maps. 
We remark that the links $L'_1, L'_2, L'_3, L'_4$ in Theorem~\ref{thm:main theorem}
are all hyperbolic of volume $10 v_{\mathrm{tet}}$, where 
$v_{\mathrm{tet}} = 1.0149 \cdots$ is the volume of the ideal regular tetrahedron. 
In fact, as was mentioned in Costantino-Thurston \cite{CT08} 
we can find a decomposition of the complement of 
each of $L'_1, L'_2, L'_3, L'_4$ into $10$ ideal regular tetrahedra. 
In the Appendix, we describe this decomposition explicitly. 

\vspace{1em}


\section{Preliminaries}
\label{sec:Preliminaries}

\vspace{1em}

Throughout the paper, we will work in the smooth category unless otherwise mentioned. 
Let $Y$ be a subspace of a polyhedral space $X$. 
The symbolds $\Nbd(Y; X)$ and $\Int (Y)$ will denote 
a regular neighborhood of $Y$ in $X$ and the interior of $Y$ in $X$, respectively. 
The number of elements of a set $S$ is denoted by $|S|$.  

\subsection{Stable maps}
\label{subsec:Stable maps}

Let $M$ be a closed orientable $3$-manifold. 
A smooth map $f$ of $M$ into $\Real^2$ is said to be {\it stable} 
if there exists an open neighborhood of $f$ in $C^\infty(M, \Real^2)$
such that for any map $g$ in that neighborhood there exist 
diffeomorphisms $\Phi : M \to M$ and $\varphi : \Real^2 \to \Real^2$ satisfying
$g=\varphi\circ f\circ \Phi^{-1}$, 
where $C^\infty(M, \Real^2)$ is the set of smooth maps of $M$ into $\Real^2$
with the Whitney $C^\infty$ topology. 
We denote by $S(f)$ the set of singular points of $f$, that is, 
$S(f) = \{ p \in M \mid \mathrm{rank}~ df_p < 2\} $. 
The stable maps form an open dense set in the space $C^\infty (M, \Real^2)$, see Mather \cite{Mat71}. 

\begin{proposition}[see e.g. Levine \cite{Lev85}]
Let $M$ be a closed-orientable $3$-manifold. 
Then a smooth map $f : M \to \Real^2$ is stable if and only if 
$f$ satisfies the following conditions $(1)$--$(6)$: 
\begin{description}
\item[Local conditions]
There exist local coordinates centered at $p$ and $f(p)$ such that 
$f$ is locally described in one of the following way: 
\begin{enumerate}
\item
$(u,x,y) \mapsto (u,x)$; 
\item
$(u,x,y) \mapsto (u,x^2 + y^2)$; 
\item
$(u,x,y) \mapsto (u,x^2 - y^2)$; 
\item
$(u,x,y) \mapsto (u,y^2 + ux -x^3)$. 
\end{enumerate}
$($In the cases of $(1)$, $(2)$, $(3)$, and $(4)$, $p$ is called a {\it regular point}, a {\it definite fold point}, 
an {\it indefinite fold point}, and a {\it cusp point}, respectively.$)$	
\item[Global conditions] 
Additionally, $f$ satisfies
\begin{enumerate}
\setcounter{enumi}{4}
\item
$f^{-1} (f(p)) \cap S(f) = \{ p \}$ for a cusp point $p$;  and
\item
the restriction of $f$ to $S(f) - \{\mbox{cusp points}\}$ is an immersion with only normal crossings.
\end{enumerate}
\end{description}
\end{proposition}
In \cite{Lev65} Levine showed that 
the cusp points of each stable map can be eliminated by 
a homotopical deformation, which implies that  
every 3-manifold admits a stable map into $\Real^2$ without
cusp points. 
We note that the set $S(f)$ forms a link in $M$. 
For general definition and properties of stable maps, see e.g. Levine \cite{Lev85}, 
Golubitsky-Guillemin \cite{GG73}, and Saeki \cite{Sae04}. 
In the following, we only consider stable maps without cusp points. 

Let $M$ be a closed orientable 3-manifold, and  $f: M \to \Real^2$ a stable map. 
We say that two points $p_1$ and $p_2$ are {\it equivalent} if they are 
contained in the same component of the fibers of $f$. 
We denote by $Q_f$ the quotient space of $M$ with respect to the equivalence relation 
and by $q_f $ the quotient map. 
We define the map $\bar{f} : Q_f \to \Real^2$ so that $f = \bar{f} \circ q_f$. 
The quotient space $Q_f$ is called the {\it Stein factorization} of $f$. 
By Kushner-Levine-Porto \cite{KLP84} and Levine \cite{Lev85} the local models of the Stein factorization 
$Q_f$ can be described as in Figure \ref{figure:stablemap_fiber}. 
\begin{figure}[htbp]
\centering\includegraphics[width=13.5cm]{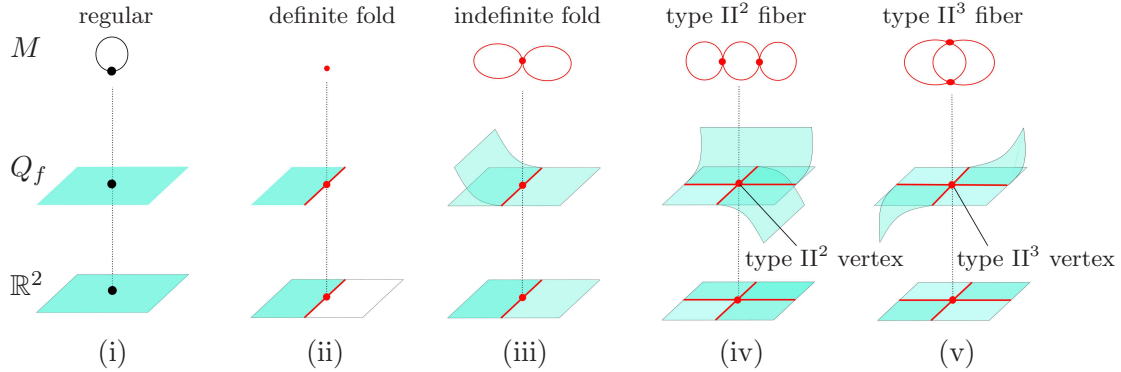}
\begin{picture}(400,0)(0,0)
\put(26,128){\footnotesize regular}
\put(98,128){\footnotesize definite fold}
\put(168,128){\footnotesize indefinite fold}
\put(248,128){\footnotesize type $\mathrm{II}^{2}$ fiber}
\put(330,128){\footnotesize type $\mathrm{II}^{3}$ fiber}
\put(278,36){\footnotesize type $\mathrm{II}^{2}$ vertex}
\put(358,36){\footnotesize type $\mathrm{II}^{3}$ vertex}
\put(0,115){$M$}
\put(0,70){$Q_{f}$}
\put(0,25){$\Real^{2}$}
\put(33,0){(i)}
\put(112,0){(ii)}
\put(186,0){(iii)}
\put(268,0){(iv)}
\put(351,0){(v)}

\end{picture}
\caption{The Stein factorization of a stable map.}
\label{figure:stablemap_fiber}
\end{figure}

As shown in the figure, there are two types of (connected components of) singular fibers 
containing exactly two indefinite points. 
The one shown in Figure (iv) ((v), respectively) is called a 
 {\it singular fiber of type $\mathrm{II}^2$} ($\mathrm{II}^3$, respectively). 
 We sometimes call a singular fiber of type $\mathrm{II}^2$ or $\mathrm{II}^3$ a 
 {\it codimension-$2$ singular fiber}. 
 We denote by $\mathrm{II}^2(f)$ and $\mathrm{II}^3(f)$ 
 the sets of singular fibers of types $\mathrm{II}^2$ and $\mathrm{II}^3$ of $f$, respectively. 
We note that both $\mathrm{II}^2(f)$ and $\mathrm{II}^3(f)$ are finite sets. 
Further, we call a vertex of the Stein factorization $Q_f$ corresponding to 
a singular fiber of  type $\mathrm{II}^2$ ($\mathrm{II}^3$, respectively) 
a vertex of {\it type $\mathrm{II}^2$} ($\mathrm{II}^3$, respectively).  

\begin{definition}
\label{def:stable maps for links}
Let $L$ be a link in a closed orientable $3$-manifold $M$. 
Then a stable map $f : M \to \Real^2$ is called a {\it stable map of $(M, L)$} if 
$L$ is contained in the set of definite fold points of $f$.   
\end{definition}
We note that for any link $L$ in a closed orientable $3$-manifold, there exists a stable map 
$f : (M, L) \to \Real^2$, see e.g. Ishikawa-Koda \cite{IK17}.

In \cite{Sae96}, Saeki gave a complete characterization of a link $L$ 
in a closed orientable $3$-manifold $M$ that admits a stable map 
$(M, L) \to \Real^2$ without codimension-$2$ singular fibers as follows. 
  
\begin{theorem}[Saeki \cite{Sae96}]
\label{thm:Saeki's theorem on graph manifolds}
Let $L$ be a link in a closed orientable $3$-manifold $M$. 
Then there exists a stable map $f : (M, L) \to \Real^2$ with 
$\mathrm{II}^2(f)  = \mathrm{II}^3(f) = \emptyset$ if and only if 
$L$ is a graph link.  
\end{theorem}

By Theorem~\ref{thm:Saeki's theorem on graph manifolds}, 
any stable map of a hyperbolic link $L$ in a closed orientable 3-manifold 
has at least one singular fiber of type $\mathrm{II}^2$ or $\mathrm{II}^3$. 
The following theorem by Ishikawa-Koda \cite[Theorem~5.6]{IK17} 
gives a complete characterization of 
the hyperbolic links in $S^3$ admitting stable maps into $\Real^2$ 
with a single singular fiber of type $\mathrm{II}^{2}$ and 
no one of type $\mathrm{II}^{3}$. 

\begin{theorem}[Ishikawa-Koda \cite{IK17}]
\label{thm:classification by Ishikawa-Koda}
Let $L$ be a hyperbolic link in $S^3$. 
Then there exists a stable map $f:(S^3,L) \to \Real^2$ with 
$|\mathrm{II}^{2}(f)| = 1$ and  $\mathrm{II}^{3}(f) = \emptyset$ if and only if 
the exterior of $L$ is diffeomorphic to
a $3$-manifold obtained by Dehn filling the exterior of one of the 
six links $L_1, L_2, \ldots, L_6$ in $S^3$ along some of 
$($possibly none of$)$ boundary tori, 
where $L_1, L_2, \ldots, L_6$ are 
depicted in Figure $\ref{figure:Ishikawa-Kodalink}$. 
\end{theorem}

\begin{figure}[htbp]
\centering\includegraphics[width=10cm]{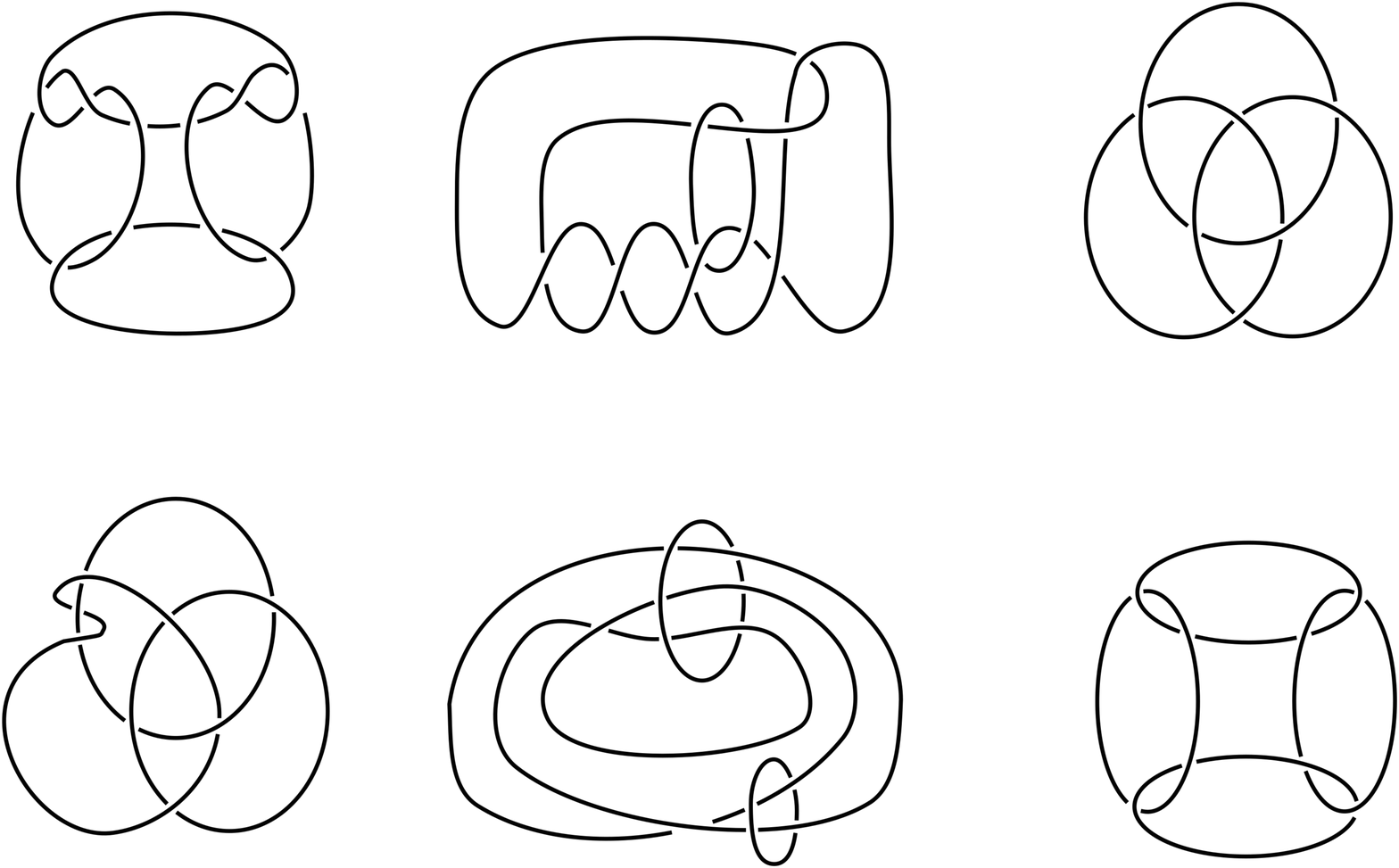}
\begin{picture}(400,0)(0,0)
\put(95,100){$L_{1}$}
\put(190,100){$L_{2}$}
\put(295,100){$L_{3}$}
\put(95,5){$L_{4}$}
\put(190,5){$L_{5}$}
\put(295,5){$L_{6}$}
\end{picture}
\caption{The links $L_{1}, L_{2}, \ldots, L_{6}$.}
\label{figure:Ishikawa-Kodalink}
\end{figure}

We note that the links $L_1, L_2, \ldots, L_6 \subset S^3$ in 
Theorem~\ref{thm:classification by Ishikawa-Koda} are all hyperbolic 
 of volume $ 2 v_{\mathrm{oct}}$, where 
$v_{\mathrm{oct}} =  3.6638 \cdots$ is the volume of the ideal regular octahedron. 
See the Appendix of this paper. 

\subsection{Branched shadows}
\label{subsec:Branched shadows} 

A compact, connected polyhedron $P$ is called 
a {\it simple polyhedron} if 
every point of $P$ has a regular neighborhood homeomorphic to 
one of the four models 
shown in Figure \ref{figure:simplepoly}. 
\begin{figure}[htbp]
\centering\includegraphics[width=12cm]{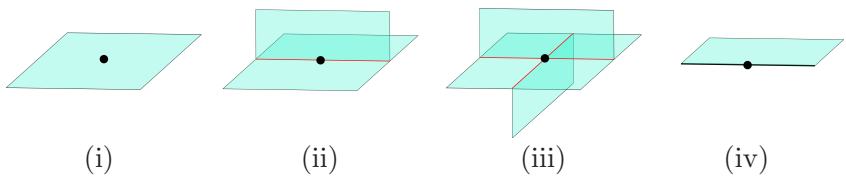}
\begin{picture}(400,0)(0,0)
\put(68,5){(i)}
\put(150,5){(ii)}
\put(233,5){(iii)}
\put(310,5){(iv)}
\end{picture}
\caption{The local models of a simple polyhedron.}
\label{figure:simplepoly}
\end{figure}

A point whose regular neighborhood is shaped on the model 
(iii) is called a {\it vertex} of $P$, and we denote the set of vertices of $P$  by $V(P)$. 
The set of points whose regular neighborhoods are shaped on the models (ii) or (iii) 
is called the {\it singular set} of $P$, 
and we denote it by $S(P)$. 
A component of $S(P) - V(P)$ is called an {\it edge} of $P$. 
We note that an edge is homeomorphic to either an open interval or a circle. 
The set of points whose regular neighborhoods are shaped on the model  (iv) is called the {\it boundary} of $P$ and 
we denote it by $\partial P$. 
Each component of $X - S(P)$ 
is called a \textit{region}  of $P$. 
A region is said to be {\it internal} if it does not touch the boundary of $P$. 

A {\it branching} of a simple polyhedron $P$ is 
the set of orientations on regions of $P$ such that 
the orientations on each edge of $P$ 
induced by the adjacent regions do not coincide.
A branching of $P$ allows us to smoothen $P$ as shown in the local models in 
Figure \ref{fig:branchedpoly}. 
\begin{figure}[htbp]
\centering\includegraphics[width=12cm]{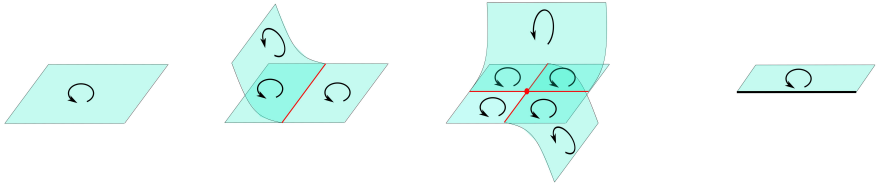}
\caption{The local models of a branched polyhedron.}
\label{fig:branchedpoly}
\end{figure}
We call a simple polyhedron equipped with a branching a {\it branched polyhedron}. 

Let $P$ be a simple polyhedron. 
A {\it coloring} of $\partial P$  is a map 
from the set of components of $\partial P$ to the set $\{ i, e, f \}$ of three elements 
$i$ (internal), $e$ (external), and $f$ (false). 
Then with respect to the coloring, $\partial P$ decomposes into three 
peaces $\partial _i P$, $\partial _e P$, and $\partial_f P$. 
A simple polyhedron is said to be {\it boundary-decorated} if 
it is equipped with a coloring of $\partial P$. 

\begin{definition}
Let $M$ be a compact orientable $3$-manifold with $($possibly empty$)$ boundary
consisting of tori. 
Let $L$ be a $($possibly empty$)$ link in $M$. 
A boundary-decorated simple polyhedron $P$ 
embedded in a compact oriented smooth $4$-manifold $W$ is called a {\it shadow} of $(M, L)$ if 
\begin{enumerate}
\item
$P$ is embedded in $W$ properly and smoothly, that is, 
$P \cap \partial W = \partial P$, and for each point $p$ in $P$, there exists a chart $(U, \varphi)$ of $W$ 
centered at $p$ 
such that $\varphi (U \cap P) \subset \{0\} \times \Real_{+}^{3} \subset \Real_{+}^{4}$ and 
$(\{0\} \times \Real_{+}^{3}, \varphi (U \cap P))$ is one of the four models in Figure \ref{figure:simplepoly} 
naturally regarded as being embedded in $\Real_{+}^{3}$. 
Here we set $\Real^{4}_{+} :=\{(x_{1},x_{2},x_{3},x_{4}) \in \Real^{4} \mid x_{4}\geq 0 \}$ and 
$\{0\} \times \Real^{3}_{+} := (\{0\} \times \Real^{3}) \cap \Real^{4}_{+} $. 
\item
$W$ collapses onto $P$ after equipping the natural PL structure on $W$; 
\item
$(M, L) = ( \partial W - \Int \Nbd (\partial_e P; \partial W),  \partial_i P)$. 
\end{enumerate}
In particular, when $P$ is a branched polyhedron, we call $P$ a {\it branched shadow} of $(M, L)$. 
\end{definition}

Under  the above definition, we sometimes call $P$ a {\it shadow} of the $4$-manifold $W$. 
We note that a shadow of $E(L)$ is obtained from that of $(M, L)$ by just replacing 
the color $i$ (internal) with $e$ (external) of the boundary-decoration of $P$, but not vice versa. 
We denote by $\pi_P$ the map $M \to P$ 
obtained by restricting the collapsing map $W \searrow P$ 
in the above definition to $M$. 
We note that when $\partial M = \emptyset$, or equivalently, when $\partial_e (P) = \emptyset$, 
$\pi_P$ is a projection. 
We can assume that the map  $\pi_P :  M \to P$ satisfies the following: 
\begin{itemize}
\item
for each $p \in \partial_i P \cup \partial_f P$, $\pi_P^{-1}(p)$ is a single point; 
\item
for each $p \in V(P)$, $\pi_P^{-1} (p) \cong T_4$, where $T_4$ is the suspension of four points; 
\item
for each $p \in S(P) - V(P)$, $\pi_P^{-1} (p) \cong T_3$, where $T_3$ is the suspension of three points; and 
\item
the restriction of $\pi_P$ to the preimage $\pi_P^{-1} (\pi_P(P) - S(P))$ is a smooth $S^1$-bundle 
over $\pi_P(P) - S(P)$. 
\end{itemize}

By Turaev \cite{Tur92, Tur94}, any pair $(M , L)$  
of a compact orientable $3$-manifold $M$ with $($possibly empty$)$ boundary
consisting of tori and a $($possibly empty$)$ link $L $ in 
$M$ has a shadow, see also Costantino \cite{Cos05}. 
As we will see in Example \ref{ex:branched shadows of links} below, when $L$ is a link in $S^3$, 
there is a standard way to obtain its branched shadow from a diagram of $L$.

Let $L$ be a $($possibly empty$)$ link in 
a compact orientable $3$-manifold $M$ with $($possibly empty$)$ boundary
consisting of tori. 
Let $P \subset W$ be a shadow of $(M, L)$. 
To each internal region $R$ of $P$, we may assign a half-integer $\gl (R)$, called a {\it gleam}, 
as follows. 
Let $\iota : R \hookrightarrow W$ be the inclusion.  
Let $\bar{R}$ be the metric completion of $R$ with the path metric inherited from a Riemannian metric on $R$. 
For sinplicity, suppose that the natural extension $\bar{\iota} : \bar{R} \to M$ is injective. 
The germs $\Nbd(\bar{R}; P) - P$ of the remaining regions near the circles $\partial \bar{R}$ 
give a structure of interval bundle over $\partial \bar{R}$, 
which is a sub-bundle of the normal bundle of $\partial \bar{R}$ in $W$. 
Let $\bar{R}'$ be a generic small perturbation of $\bar{R}$ in $W$ 
such that $\partial \bar{R}'$ lies in the interval
bundle. 
The gleam $\gl (R)$ is then (well-)defined by counting the finitely many 
isolated intersections of $\bar{R}$ and $\bar{R}'$ with signs 
as follows: 
\[
\gl (R) = \frac{1}{2} | \partial \bar{R} \cap \partial \bar{R}' | + | \Int\bar{R} \cap \Int\bar{R}' | \in \frac{1}{2} \Integer . 
\]
We call a polyhedron $P$ equipped with a gleam on each internal region a {\it shadowed polyhedron}.  
In \cite{Tur92, Tur94}, Turaev showed that the 4-manifold $W$, the 3-manifold $M \subset \partial W$ and 
the link $L \subset M$ are all recovered from a shadowed, boundary-decorated polyhedron $P$ in a canonical way.

\begin{example}
\label{ex:branched shadows of links}
Let $L = L_1 \sqcup L_2 \sqcup \cdots \sqcup L_n$ be an $n$-component link in $S^3$. 
Using a diagram of $L$, we can construct a shadow of $(S^3, L)$ as in the following standard way. 

We think of $S^3$ as the boundary of the $4$-ball 
$D^4 = \{ (z, w) \in \Complex^2 \mid |z|^2 + |w|^2 \leq 1 \} $. 
Set $D := \{ (z, w) \in D^4 \mid w=0 \} $, which is an oriented $2$-disk. 
Let $D'$ be the closure of $D - \Nbd (\partial D; D) $. 
Let $\pi : S^3 \to D$ be the projection induced by a collapsing $D^4 \searrow D$. 
We can assume without loss of generality that 
the preimage $V = \pi^{-1} (D')$ is a solid torus, which we identify with $D' \times S^1$. 
We note that $S^3 - \Int V$ is also a solid torus. 
We can push $L$ by isotopy so that $L$ lies in $D' \times [[-\varepsilon], [\varepsilon]]$, where 
$S^1$ is identified with $\Real / 2 \pi \Integer$,  
$\varepsilon$ is a sufficiently small positive real number, and 
we assume that $\pi |_{V}$ is generic with respect to 
$L$. 
Thus, the image $\pi (L)$, together with over/under crossing information at each double point 
of $\pi (L)$, gives a 
diagram $D_L$ of $L$ on $D'$. 
Consider the mapping cylinder 
\[ P_{D_L}^* = ( (L \times [0,1]) \sqcup D' ) / (x, 0)  \sim \pi (x) . \] 
Then $P_{D_L}^*$, together with the color $e$ (external) of each $L_j \times \{ 1 \} \subset \partial P_L^*$ 
and the color $f$ (false) of the remaining boundary circle, 
is a shadow of $(S^3, L)$. 
The set of the regions of $P_{D_L}^*$ consists of 
subsurfaces of $D'$ and half-open annuli $A_j = L_j \times (0,1]$, $j \in \{ 1,2, \ldots, n\}$. 
The gleam on each internal region of $P_{D_L}^*$, which lies in $D'$, is obtained by counting 
the local contributions around crossings as indicated in Figure \ref{figure:example_PDHstar}. 
Refer to Turaev \cite{Tur94} and Costantino \cite{Cos05} for more details. 

\begin{figure}[htbp]
\centering\includegraphics[width=12cm]{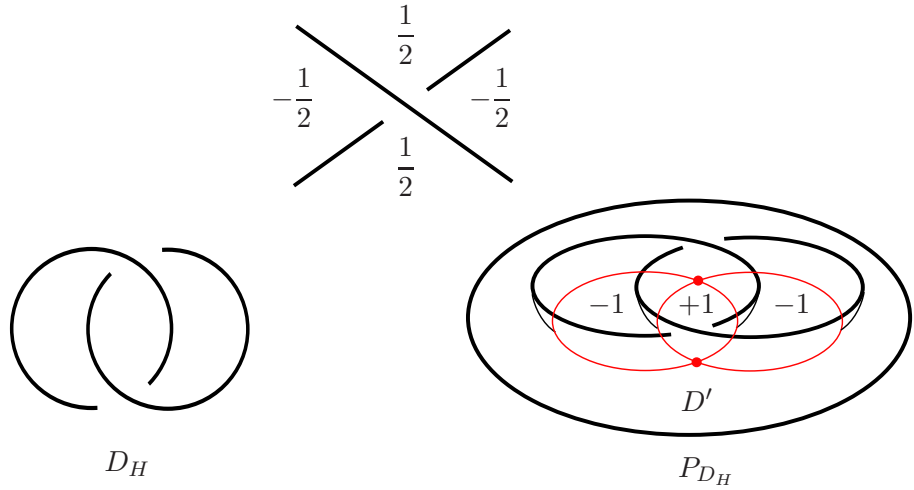}
\begin{picture}(400,0)(0,0)
\put(67,0){$D_{H}$}
\put(284,-2){$P_{D_{H}}$}

\put(205,138){$-\dfrac{1}{2}$}
\put(130,138){$-\dfrac{1}{2}$}
\put(177,162){$\dfrac{1}{2}$}
\put(177,112){$\dfrac{1}{2}$}

\put(285,23){$D'$}

\put(250,60){$-1$}
\put(284,60){$+1$}
\put(320,60){$-1$}

%
%
%
\end{picture}
\setlength\abovecaptionskip{15pt}
\caption{The diagram $D_H$ of the Hopf link $H$ and the branched shadow $P_{D_H}^*$.}
\label{figure:example_PDHstar}
\end{figure}

To each region of $P_{D_L}^*$ contained in $D'$, 
we give the orientation induced by the prefixed orientation of $D'$.  
By equipping, further, orientations to the other regions $A_1, A_2, \ldots, A_n$ in an arbitrary way, 
$P_{D_L}^*$ becomes a branched shadow of $(S^3, L)$. 
Let $R$ be the region of $P_{D_L}^*$ touching $\partial D'$. 
Suppose that the branched polyhedron $P_{D_L}^*$ satisfies the following: 
\begin{itemize}
\item
the closure of $R$ is an annulus; and 
\item
on each edge of $P_{D_L}^*$ touched by $R$ and $A_j$ ($j \in \{ 1, 2, \ldots, n\}$), 
the orientations induced from those of 
$R$ and $A_j$ coincide. 
\end{itemize}
We note that if we choose, for example, $D_L$ to be a closed braid presentation of $L$, 
we can find a branching of $P_{D_L}^*$ satisfying the above conditions. 
Then, $P_{D_L} := P_{D_L}^* - R$ is still a branched polyhedron, 
thus, a branched shadow of $(S^3, L)$. 
See Figure \ref{figure:example_from_PDHstar_to_PHD}. 
\begin{figure}[htbp]
\centering\includegraphics[width=11cm]{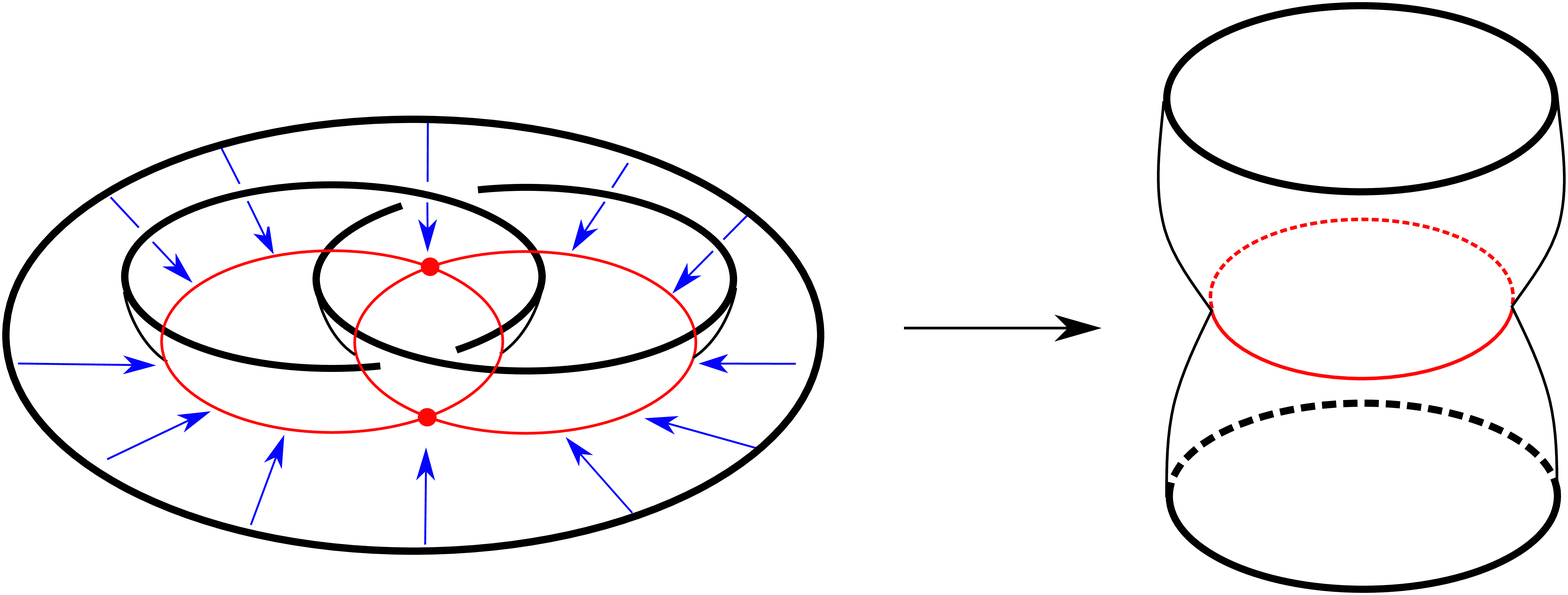}
\begin{picture}(400,0)(0,0)
\put(120,0){$D_{H}$}

\put(89,67){$-1$}
\put(123,67){$+1$}
\put(157,67){$-1$}
\put(110,32){$R$}

\put(307,-2){$P_{D_{H}}$}
\put(311,70){$+1$}

%
%
\end{picture}
\setlength\abovecaptionskip{15pt}
\caption{From $P_{D_H}^*$ to $P_{D_H}$.}
\label{figure:example_from_PDHstar_to_PHD}
\end{figure}
\end{example}

Let $M$ be a compact orientable $3$-manifold $M$ with $($possibly empty$)$ boundary
consisting of tori. 
Let $P \subset W$ be a shadow of $M = (M, \emptyset)$. 
Suppose for simplicity that all components of $\partial P$ is colored by $i$ (internal), and 
no component of $S(P)$ is a circle. 
Then the combinatorial structure of $P$ induces a decomposition of $P$ into pieces 
each of which is homeomorphic to one of 
\begin{enumerate}
\item
the model in Figure \ref{figure:simplepoly} (iii); 
\item
the model in Figure \ref{figure:simplepoly} (ii); or
\item
a compact surface $F$.
\end{enumerate}
Here, the pieces of (1) one-to-one correspond to the vertices of $P$,  
the pieces of (2) one-to-one correspond to the edges of $P$, and   
the pieces of (3) one-to-one correspond to the regions of $P$. 
We call this the {\it combinatorial decomposition} of $P$. 
The combinatorial decomposition then induces, through the 
preimage of the map $\pi_P: M \to P$, a decomposition of $M$ into pieces 
each of which is homeomorphic to one of 
\begin{enumerate}
\item
a handlebody of genus three; 
\item
a handlebody of genus two; or 
\item
an orientable $S^1$ bundle over a compact surface.
\end{enumerate}
We note that each piece of $(2)$ here is naturally equipped with the 
product structure $R \times [0,1]$, where $R$ is a pair of pants, 
through $\pi_P^{-1}$. 

\section{Relationship between Stein factorizatons and branched shadows}
\label{sec:Relationship between Stein factorizatons and branched shadows}

From now until the end of the paper, $X$ always denotes the local part of a branched, shadowed polyhedron depicted in 
Figure \ref{figure:X}, where $+1$ in the figure is the gleam on the bigon.

\begin{figure}[htbp]
\centering\includegraphics[width=8cm]{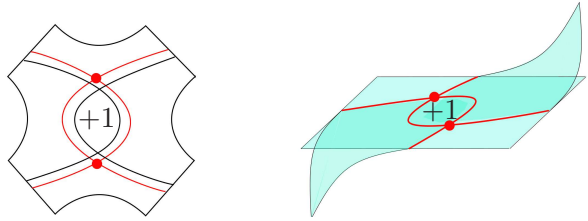}
\begin{picture}(400,0)(0,0)
\put(115,47){+1}
\put(245,50){+1}
\end{picture}
\caption{The local part $X$ of a branched, shadowed polyhedron.}
\label{figure:X}
\end{figure}

Let $M$ be a closed orientable $3$-manifold. 
As we have seen in Section \ref{subsec:Stable maps}, 
each point of the Stein factorization of a stable map of $M$ into $\Real^2$ is homeomorphic to one of the four local models shown in Figure 
$\ref{figure:steindecomodel}$. 
\begin{figure}[htbp]
\centering\includegraphics[width=12cm]{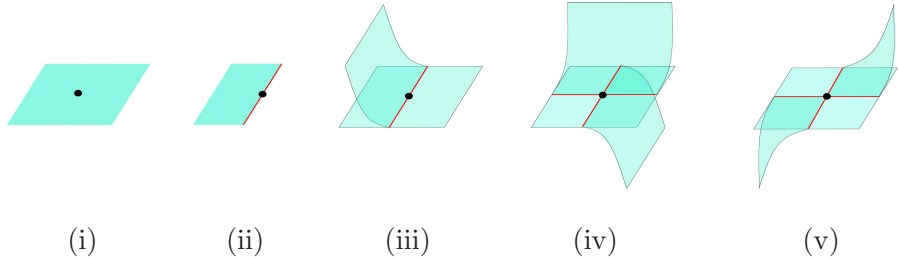}
\begin{picture}(400,0)(0,0)
\put(55,5){(i)}
\put(115,5){(ii)}
\put(175,5){(iii)}
\put(246,5){(iv)}
\put(333,5){(v)}
\end{picture}
\caption{The local models of a Stein factorization.}
\label{figure:steindecomodel}
\end{figure}
From the figure, we see that the Stein factorization is ``almost" a branched polyhedron.  
Indeed, the local models in Figure \ref{figure:steindecomodel} coincide with 
those in Figure \ref{fig:branchedpoly} except 
the model of Figure \ref{figure:steindecomodel} (iv), which is a regular  
neighborhood of a vertex of type $\mathrm{II}^3$. 
Costantino-Thurston \cite[Section 4]{CT08} revealed that 
we can actually obtain a branched shadow from 
the Stein factorization of a stable map by replacing a regular  
neighborhood of each vertex of type $\mathrm{II}^3$ with $X$.  

\begin{theorem}[Costantino-Thurston \cite{CT08}]
\label{thm:from a Stein factorization to a branched shadow}
Let $L$ be a link in a closed orientable $3$-manifold $M$. 
Let $f: (M, L) \to \Real^2$ be a stable map. 
Then a branched polyhedron $P$ obtained by replacing a regular  
neighborhood of each vertex of type $\mathrm{II}^3$ of $Q_f$ with 
$X$ as shown in Figure $\ref{figure:II3X}$ is a branched shadow of $(M, L)$. 
\end{theorem}
\begin{figure}[htbp]
\centering\includegraphics[width=8cm]{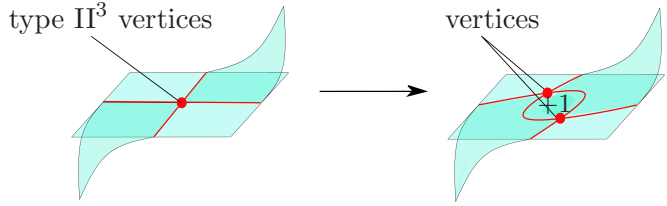}
\begin{picture}(400,0)(0,0)
\put(265,48){{\small $+1$}}
\put(65,80){type $\mathrm{II}^{3}$ vertices}
\put(230,80){vertices}
\end{picture}
\caption{Replacement of a neighborhood of type $\mathrm{II}^{3}$ vertex with $X$.}
\label{figure:II3X}
\end{figure}

In fact, in Theorem~\ref{thm:from a Stein factorization to a branched shadow}
the $4$-manifold $W$ containing $P$ as a shadow can be constructed 
from $f$ in a natural way, and $Q_f$ is already naturally embedded in $W$.  
The replacement in Theorem~\ref{thm:from a Stein factorization to a branched shadow} 
can then be performed in $W$. 

The following theorem by Ishikawa-Koda 
\cite[Theorem~3.5, Lemma~3.12 and Corollary~3.13]{IK17} is a sort of inverse of 
Theorem~ \ref{thm:from a Stein factorization to a branched shadow}: 
we can construct a stable map from a branched shadow.

\begin{theorem}[Ishikawa-Koda \cite{IK17}]
\label{thm:bsc is equal to smc}
Let $L$ be a link in a closed orientable $3$-manifold $M$. 
Let $P$ be a branched shadow of $E(L)$. 
Then there exists a stable map $f : (M,L) \to \Real^{2}$ with 
$|\mathrm{II}^{2}(f)| = |V(P)|$ and $\mathrm{II}^{3}(f) = \emptyset$. 
\end{theorem}
The proof of Theorem~\ref{thm:bsc is equal to smc} given in \cite{IK17} is constructive. 
We remark that the Stein factorization $Q_f$ of the stable map $f$ in Theorem~\ref{thm:bsc is equal to smc} 
is possibly not homeomorphic to the given branched polyhedron $P$ in general. 
However, if we regard $Q_f$ as a branched polyhedron 
(this is possible because $\mathrm{II}^{3}(f) = \emptyset$), 
there exists a canonical embedding of $\Nbd (S(P) ; P)$ into 
 $\Nbd (S(Q_f) ; Q_f)$. 
Further, if $P$ contains $X$, the Stein factorization $Q_f$ also contains $X$ at the corresponding part.  

\begin{example}
\label{ex:stable maps of links}

Let  $P_{D_{H}}^*$ is a branched shadow of the Hopf link $H$ 
shown in Figure \ref{figure:example_PDHstar} in Example \ref{ex:branched shadows of links}. 
Let $h : (S^3,H) \to \Real^2$ be a stable map constructed from $P_{D_{H}}^*$ by 
Theorem~\ref{thm:bsc is equal to smc}. 
Then we can identify the Stein factorization $Q_h$ with $P_{D_{H}}^*$. 
The fiber $q_{h}^{-1} (p)$ over each point $p \in P_{D_{H}}^*$ can be described as follows. 
Here, we recall that $q_{h}: S^{3} \to P_{D_{H}}^*$ is the 
quotient map. 
We use the same notation as in Example \ref{ex:branched shadows of links}. 
 
\begin{itemize}

\item 
If $p \in D'-S(P_{D_{H}}^*)$, then $q_{h}^{-1}(p) = \{p\} \times S^{1} \subset D' \times S^1 = V$. 
See Figure \ref{figure:hfiber_ryouiki}. 
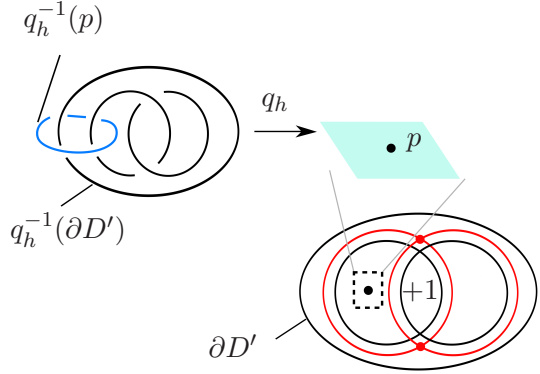
\begin{figure}[htbp]
\centering\includegraphics[width=7cm]{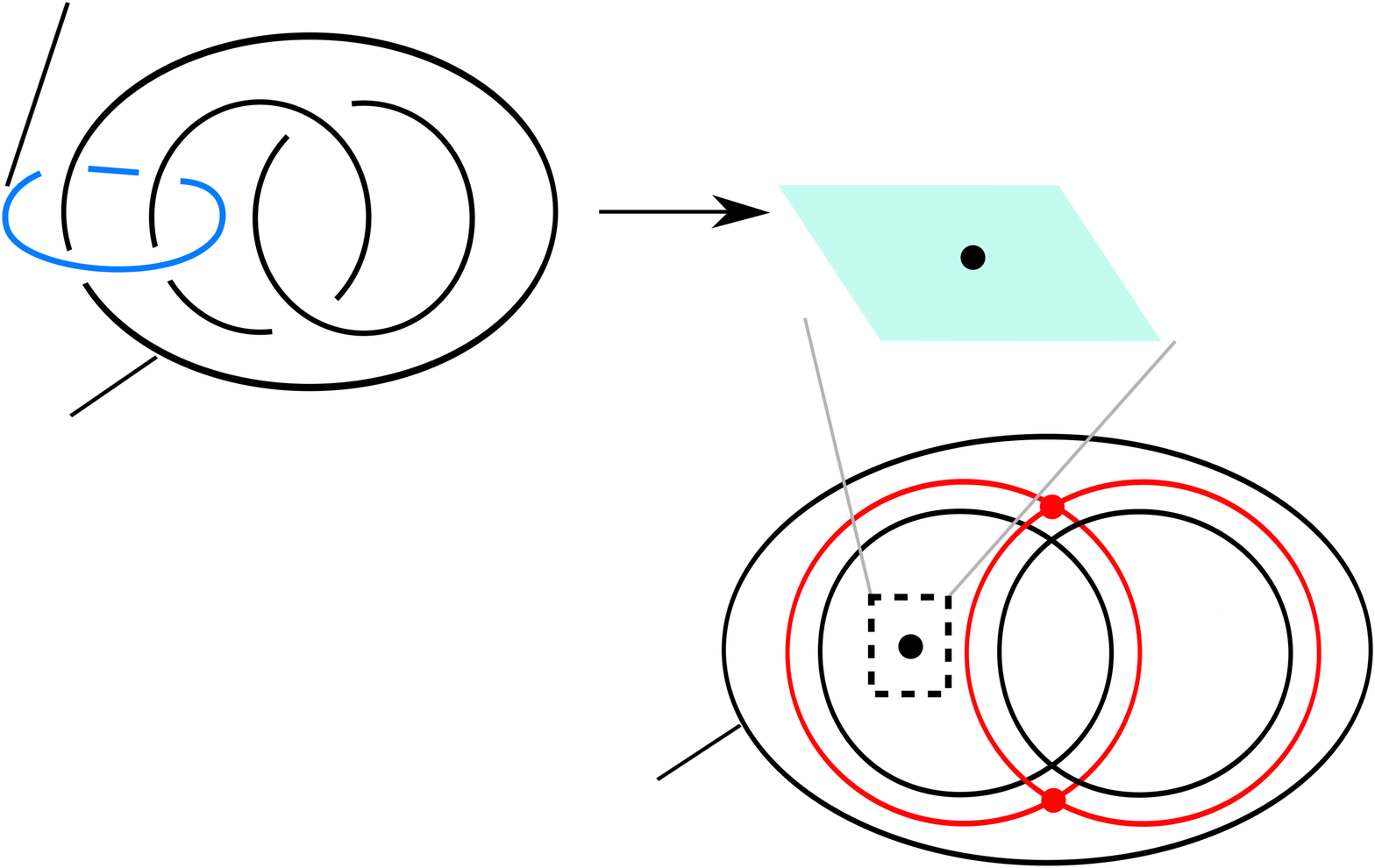}
\begin{picture}(400,0)(0,0)
\put(100,145){$q_{h}^{-1}(p)$}
\put(95,65){$q_{h}^{-1}(\partial D')$}
\put(245,100){$p$}
\put(190,115){$q_{h}$}
\put(243,42){$+1$}
\put(170,20){$\partial D'$}
\end{picture}
\caption{The fiber of $q_h$ over $p \in D'-S(P_{D_{H}}^{\ast})$.}
\label{figure:hfiber_ryouiki}
\end{figure}

\item 
Let $K$ be a component of $H$. 
Let $p \in K \times \{1\} \subset P_{D_{H}}^*$ and 
$p_{1} \in K \times (0,1) \subset P_{D_{H}}^*$. 
Then $q_{h}^{-1}(p)$ is a point on $K$, while $q_{h}^{-1}(p_{1})$ is a meridian of $K$. 
See Figure \ref{figure:hfiber_K}. 
\begin{figure}[htbp]
\centering\includegraphics[width=7cm]{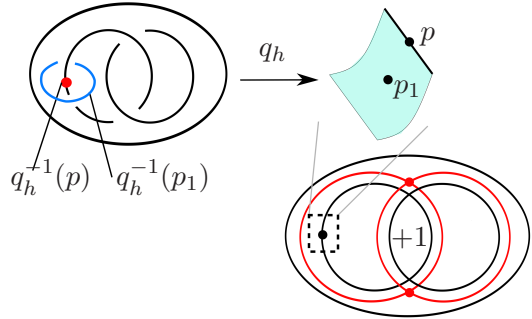}
\begin{picture}(400,0)(0,0)
\put(140,65){$q_{h}^{-1}(p_{1})$}
\put(100,65){$q_{h}^{-1}(p)$}
\put(194,114){$q_{h}$}
\put(245,100){$p_{1}$}
\put(244,42){$+1$}
\put(255,120){$p$}
\end{picture}
\caption{The fibers of $q_h$ over $p \in K \times \{1\} $ and 
$p_{1} \in K \times (0,1)$.}
\label{figure:hfiber_K}
\end{figure}

\item 
Let $p \in S(P_{D_{H}}^*)-V(P_{D_{H}}^*)$. 
Choose three points $p_1, p_2, p_3 \in P_{D_{H}}^*$ near $p$ as in Figure 
\ref{figure:hfiber_indefinite}. 
Then from the configuration of the fibers 
$q_{h}^{-1}(p_{1})$, $q_{h}^{-1}(p_{2})$, $q_{h}^{-1}(p_{3})$, we can find the 
configuration of $q_{h}^{-1}(p)$ as shown in Figure \ref{figure:hfiber_indefinite}.
\begin{figure}[htbp]
\centering\includegraphics[width=12cm]{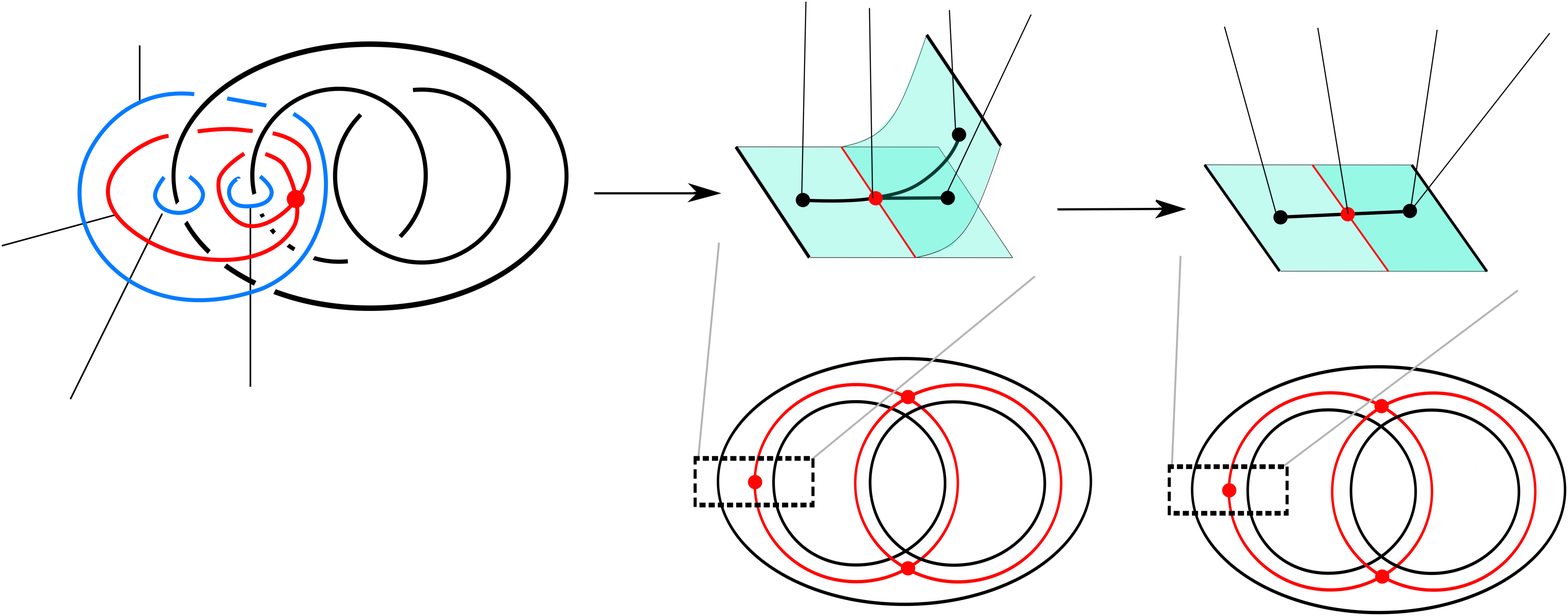}
\begin{picture}(400,0)(0,0)
\put(50,142){$q_{h}^{-1}(p_{3})$}
\put(5,85){$q_{h}^{-1}(p)$}
\put(30,53){$q_{h}^{-1}(p_{1})$}
\put(75,53){$q_{h}^{-1}(p_{2})$}
\put(192,149){$p_{1}$}
\put(208,149){$p$}
\put(223,149){$p_{2}$}
\put(242,149){$p_{3}$}

\put(155,112){$q_{h}$}
\put(250,110){$\bar{f}$}

\put(263,147){$\bar{h}(p_{1})$}
\put(293,147){$\bar{h}(p)$}
\put(321,147){$\bar{h}(p_{2})$}
\put(353,147){$\bar{h}(p_{3})$}

\put(211,41){$+1$}
\end{picture}
\caption{The fiber of $q_f$ over $p \in S(P_{D_{H}}^*)-V(P_{D_{H}}^*)$.}
\label{figure:hfiber_indefinite}
\end{figure}

\item 
Let $p \in V(P_{D_{H}}^*)$. 
Choose three points $p_1, p_2, p_3, p_4 \in P_{D_{H}}^*$ near $p$ as in Figure 
\ref{figure:hfiber_vertex}. 
Then from the configuration of the fibers 
$q_{h}^{-1}(p_{1})$, $q_{h}^{-1}(p_{2})$, $q_{h}^{-1}(p_{3})$, $q_{h}^{-1}(p_{4})$, we can find the 
configuration of $q_{h}^{-1}(p)$ as shown in Figure \ref{figure:hfiber_vertex}.
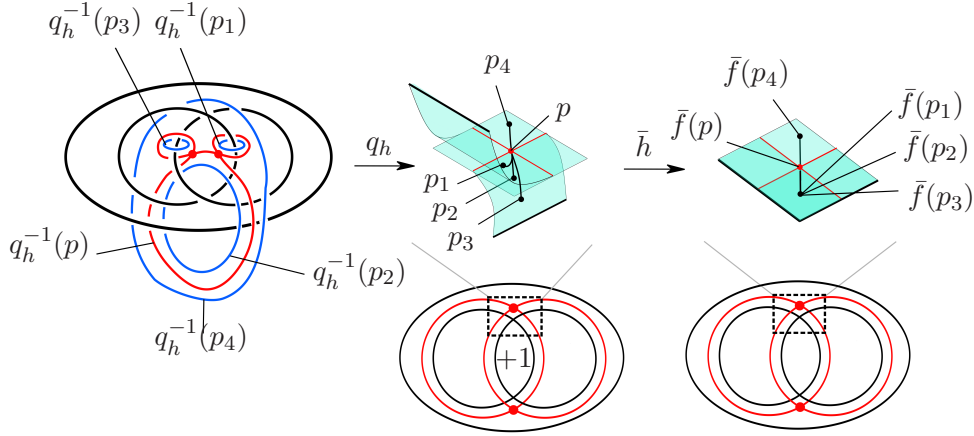
\begin{figure}[htbp]
\centering\includegraphics[width=12cm]{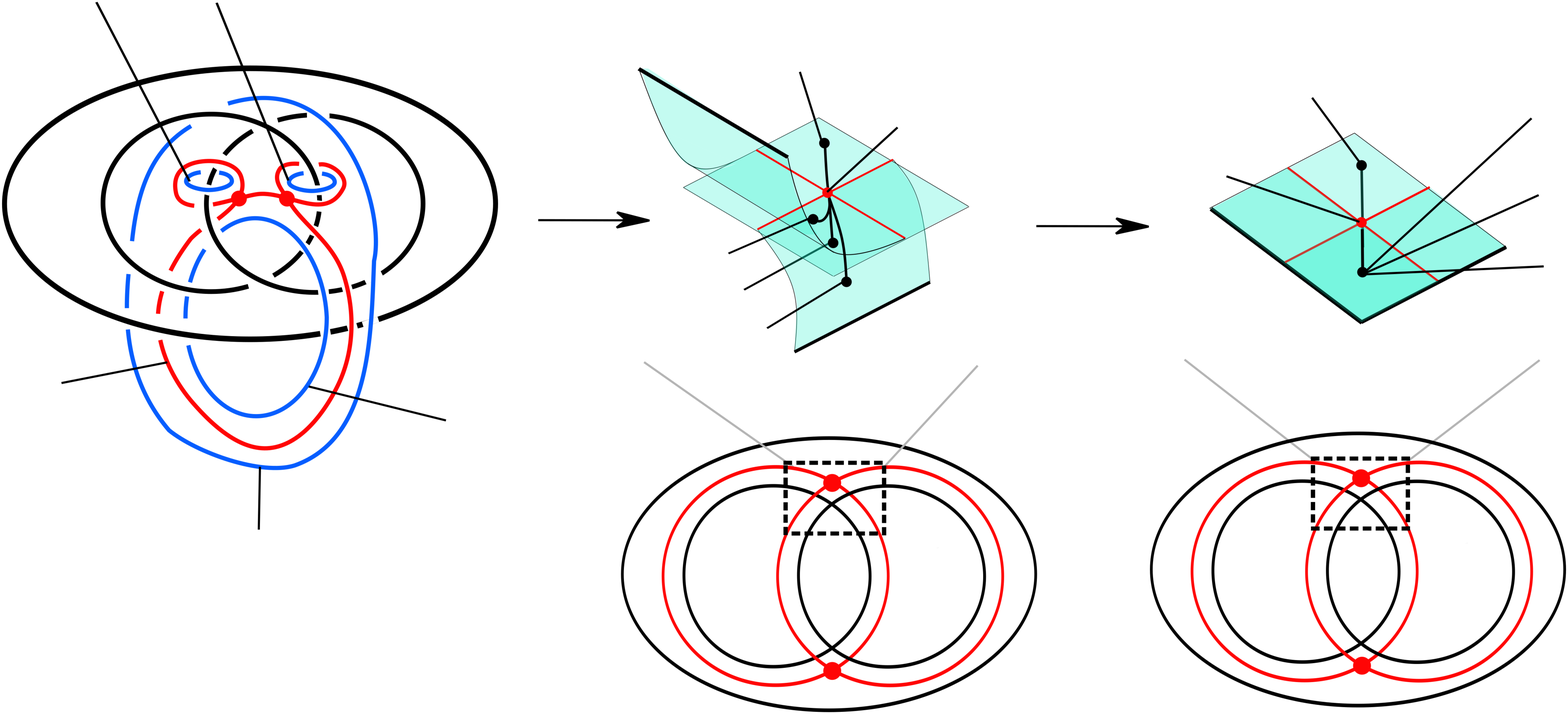}
\begin{picture}(400,0)(0,0)
\put(75,50){$q_{h}^{-1}(p_{4})$}
\put(35,170){$q_{h}^{-1}(p_{3})$}
\put(75,170){$q_{h}^{-1}(p_{1})$}
\put(20,85){$q_{h}^{-1}(p)$}
\put(135,75){$q_{h}^{-1}(p_{2})$}
\put(177,110){$p_{1}$}
\put(180,99){$p_{2}$}
\put(186,88){$p_{3}$}
\put(200,155){$p_{4}$}
\put(227,137){$p$}

\put(155,124){$q_{h}$}
\put(257,122){$\bar{h}$}

\put(270,130){$\bar{f}(p)$}
\put(290,150){$\bar{f}(p_{4})$}
\put(356,138){$\bar{f}(p_{1})$}
\put(358,122){$\bar{f}(p_{2})$}
\put(360,103){$\bar{f}(p_{3})$}

\put(204,42){$+1$}
\end{picture}

\caption{The fiber of $q_f$ over $p \in V(P_{D_{H}}^*)$.}
\label{figure:hfiber_vertex}
\end{figure}

\end{itemize}

\end{example}

\begin{corollary}
\label{cor:replacement stable maps}
Let $L$ be a link in a closed orientable $3$-manifold $M$. 
Let $P$ be a branched shadow of $E(L)$. 
Let $f: (M, L) \to \Real^2$ be a stable map obtained from $P$ as 
in Theorem~$\ref{thm:bsc is equal to smc}$. 
Then there exists a stable map $g: (M, L) \to \Real^2$ such that 
the Stein factorization $Q_g$ is obtained from $Q_f$ by replacing 
each part homeomorphic to $X$ with 
the model of Figure $\ref{figure:steindecomodel}$ {\rm (iv)}. 
\end{corollary}
\begin{figure}[htbp]
\centering\includegraphics[width=8cm]{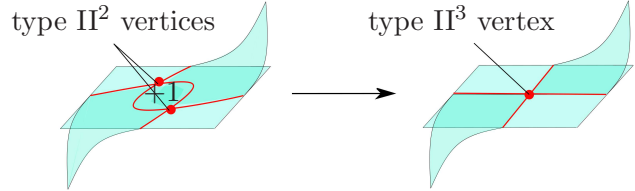}
\begin{picture}(400,0)(0,0)
\put(125,53){$+1$}
\put(75,80){type $\mathrm{II}^{2}$ vertices}
\put(210,80){type $\mathrm{II}^{3}$ vertex}
\end{picture}
\caption{Replacement of $X$ with a neighborhood of type $\mathrm{II}^{3}$ vertex.}
\end{figure}
\begin{proof}
Suppose that $Q_f$ contains a local part, which we simply denote by $X$, homeomorphic to $X$.  
As we have seen in Example \ref{ex:stable maps of links}, 
the fiber of $q_f$ over each point of the boundary of $X$ is as shown in Figure \ref{figure:Xfiber}. 
\begin{figure}[htbp]
\centering\includegraphics[width=12cm]{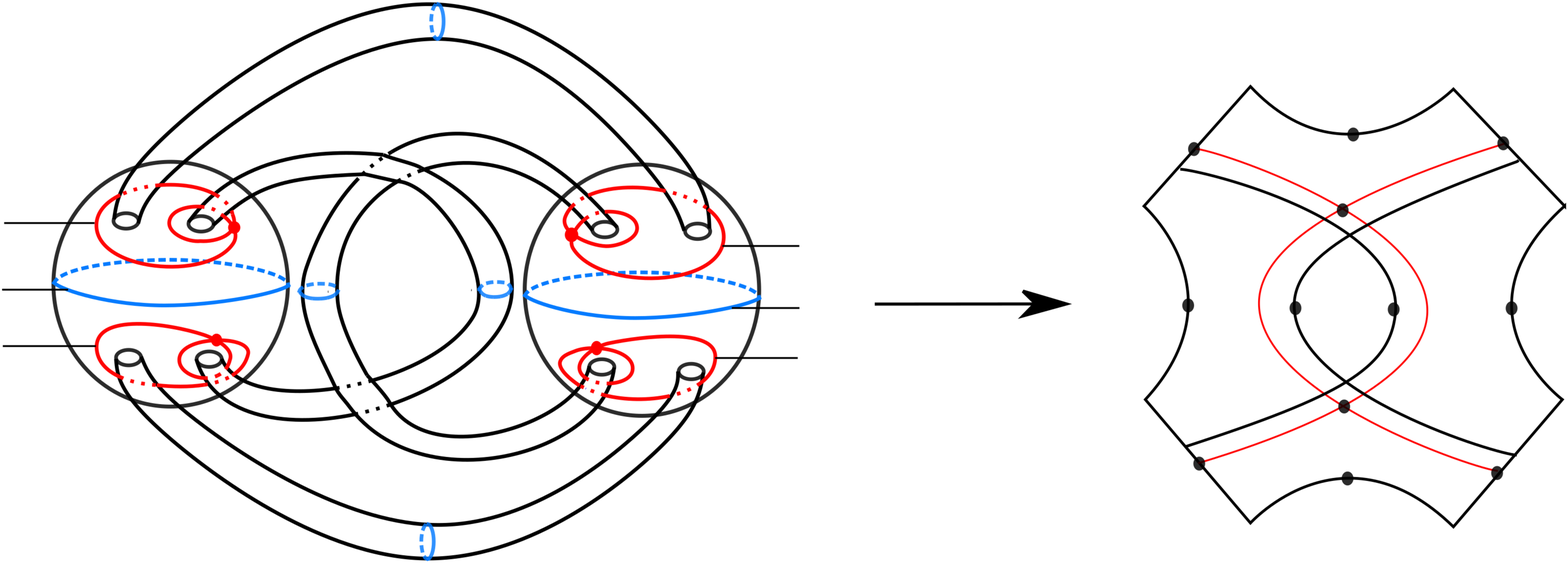}
\begin{picture}(400,0)(0,0)
\put(25,130){$S^{3}$}
\put(20,5){$q_f^{-1}(X)$ is the exterior of this handlebody.}
\put(308,5){$X$}

\put(199,90){$\textcolor{red}4$}
\put(199,78){$\textcolor{blue}5$}
\put(199,67){$\textcolor{red}2$}
\put(119,126){$\textcolor{blue}1$}
\put(117,38){$\textcolor{blue}4$}
\put(103,80){$\textcolor{blue}2$}
\put(121,80){$\textcolor{blue}3$}
\put(20,95){$\textcolor{red}1$}
\put(20,80){$\textcolor{blue}6$}
\put(20,67){$\textcolor{red}3$}

\put(226,90){$q_{f}$}

\put(350,118){$\textcolor{red}4$}
\put(351,78){$\textcolor{blue}5$}
\put(350,40){$\textcolor{red}2$}

\put(311,121){$\textcolor{blue}1$}
\put(295,78){$\textcolor{blue}2$}
\put(324,78){$\textcolor{blue}3$}
\put(310,35){$\textcolor{blue}4$}

\put(272,118){$\textcolor{red}1$}
\put(270,78){$\textcolor{blue}6$}
\put(272,40){$\textcolor{red}3$}



\put(305,77){$+1$}
\end{picture}
\caption{The fiber of $q_f$ over each point of the boundary of $X$.}
\label{figure:Xfiber}
\end{figure}

Let $f': M \to \Real^2$ be a stable map containing a singular fiber of type $\mathrm{II}^3$. 
Let $v$ be a type $\mathrm{II}^{3}$ vertex of the Stein factorization $Q_{f'}$. 
Then the behavior of the map $f'$ on the preimage of $\Nbd (v; Q_{f'})$ under the quotient map $q_{f'}$ 
is as drawn in Figure \ref{figure:stein_II3}. 
\begin{figure}[htbp]
\centering\includegraphics[width=13cm]{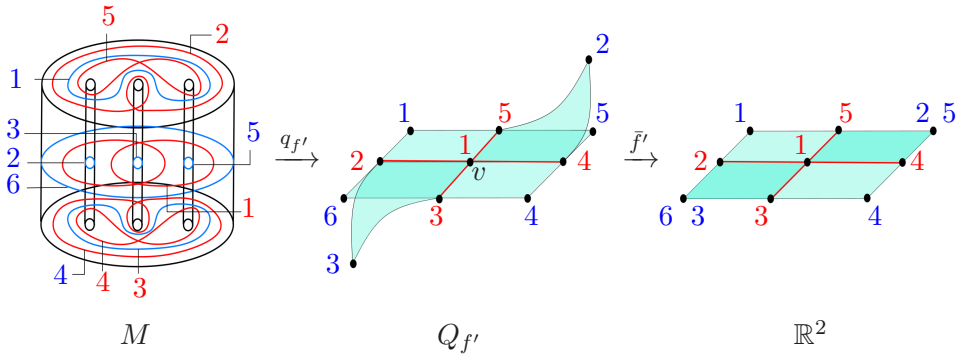}
\begin{picture}(400,0)(0,0)
\put(60,10){$M$}
\put(180,10){$Q_{f'}$}
\put(315,10){$\Real^{2}$}

\put(118,77){$\xrightarrow{q_{f'}}$}
\put(250,77){$\xrightarrow{\bar{f'}}$}

\put(193,72){$v$}

\put(53,131){$\textcolor{red}5$}
\put(96,124){$\textcolor{red}2$}

\put(18,107){$\textcolor{blue}1$}
\put(17,88){$\textcolor{blue}3$}
\put(17,77){$\textcolor{blue}2$}
\put(17,67){$\textcolor{blue}6$}

\put(35,33){$\textcolor{blue}4$}
\put(51,30){$\textcolor{red}4$}

\put(65,28){$\textcolor{red}3$}

\put(106,58){$\textcolor{red}1$}

\put(108,87){$\textcolor{blue}5$}

\put(240,120){$\textcolor{blue}2$}

\put(165,94){$\textcolor{blue}1$}
\put(203,94){$\textcolor{red}5$}
\put(240,94){$\textcolor{blue}5$}

\put(147,76){$\textcolor{red}2$} 
\put(188,82){$\textcolor{red}1$}
\put(233,76){$\textcolor{red}4$}

\put(137,55){$\textcolor{blue}6$}
\put(177,55){$\textcolor{red}3$}
\put(213,55){$\textcolor{blue}4$}
\put(138,37){$\textcolor{blue}3$}

\put(292,94){$\textcolor{blue}1$}
\put(332,94){$\textcolor{red}5$}
\put(361,94){$\textcolor{blue}2$}
\put(371,94){$\textcolor{blue}5$}

\put(277,76){$\textcolor{red}2$}
\put(315,82){$\textcolor{red}1$}
\put(359,76){$\textcolor{red}4$}

\put(264,55){$\textcolor{blue}6$}
\put(276,55){$\textcolor{blue}3$}
\put(300,55){$\textcolor{red}3$}
\put(343,55){$\textcolor{blue}4$}
\end{picture}
\caption{The Stein factorization of $f'|_{q_{f'}^{-1} (\Nbd (v; Q_{f'}))}$.}
\label{figure:stein_II3}
\end{figure}

We can assume that $\bar{f} ( X ) = \bar{f'} (\Nbd (v;Q_{f'}))$.  
We identify the boundary of $X$ and that of $\Nbd (v;Q_{f'})$ naturally so that  
 the map $\bar{f}$ coincides with $\bar{f'}$ on the boundaries modulo the identification. 
We note that both $q_{f}^{-1}(X)$ and $q_{f'}^{-1}(\Nbd (v;Q_{f'}))$ are handlebodies of genus three. 
We define a diffeomorphism from $q_{f}^{-1}(X)$ to $q_{f'}^{-1}(\Nbd (v;Q_{f'}))$  
using the handle slide shown on the bottom in Figure \ref{figure:handleslide}. 
We then identify the two handlebodies by this diffeomorphism and denote both of them by $V$. 
Under this identification, the fiber of $q_f$ over each point of the boundary of $X$ 
is identified with that of $q_{f'}$ over a point of the boundary of $\Nbd (v;Q_{f'})$ as indicated by the labels in  
Figure \ref{figure:handleslide}. 
\begin{figure}[htbp]
\centering
\includegraphics[width=13cm]{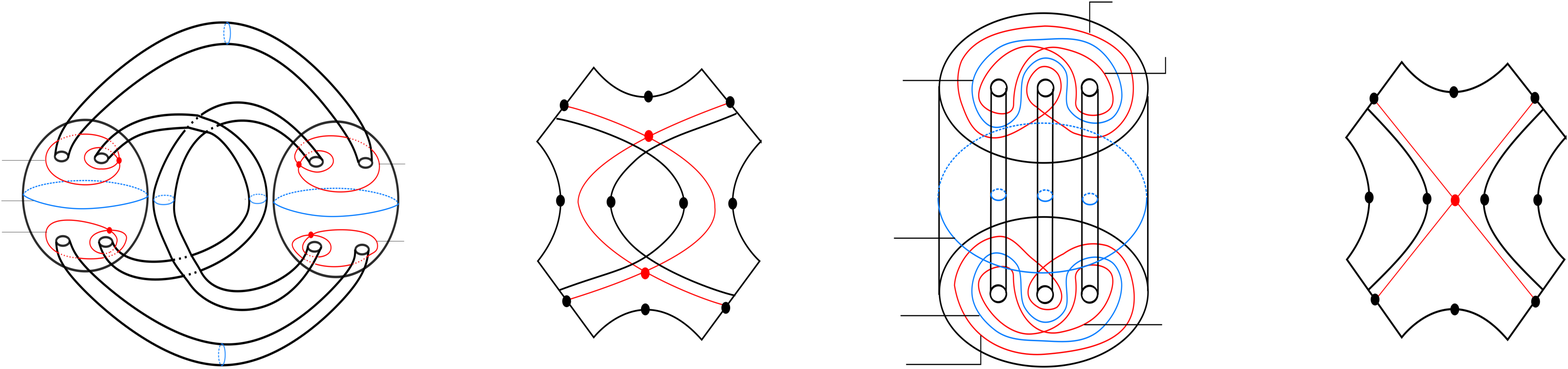}
\begin{picture}(400,0)(0,0)
\put(121,59){$\xrightarrow{q_{f}}$}
\put(302,59){$\xrightarrow{q_{f'}}$}
\put(161,57){$+1$}

\put(12,67){$\textcolor{red}1$}
\put(12,58){$\textcolor{blue}6$}
\put(12,49){$\textcolor{red}3$}
\put(67,88){$\textcolor{blue}1$}
\put(59,58){$\textcolor{blue}2$}
\put(68,58){$\textcolor{blue}3$}
\put(65,13){$\textcolor{blue}4$}
\put(112,67){$\textcolor{red}4$}
\put(112,58){$\textcolor{blue}5$}
\put(112,49){$\textcolor{red}2$}

\put(140,85){$\textcolor{red}1$}
\put(140,57){$\textcolor{blue}6$}
\put(140,32){$\textcolor{red}3$}
\put(165,90){$\textcolor{blue}1$}
\put(153,57){$\textcolor{blue}2$}
\put(177,57){$\textcolor{blue}3$}
\put(165,26){$\textcolor{blue}4$}
\put(189,85){$\textcolor{red}4$}
\put(190,57){$\textcolor{blue}5$}
\put(189,32){$\textcolor{red}2$}

\put(220,85){$\textcolor{blue}1$}
\put(220,49){$\textcolor{blue}6$}
\put(220,31){$\textcolor{blue}4$}
\put(220,19){$\textcolor{red}3$}
\put(242,58){$\textcolor{blue}2$}
\put(253,58){$\textcolor{blue}3$}
\put(263,58){$\textcolor{blue}5$}
\put(276,104){$\textcolor{red}1$}
\put(286,95){$\textcolor{red}4$}
\put(290,29){$\textcolor{red}2$}

\put(328,85){$\textcolor{red}1$}
\put(328,57){$\textcolor{blue}6$}
\put(328,32){$\textcolor{red}3$}
\put(353,90){$\textcolor{blue}1$}
\put(343,57){$\textcolor{blue}3$}
\put(365,57){$\textcolor{blue}2$}
\put(353,26){$\textcolor{blue}4$}
\put(377,85){$\textcolor{red}4$}
\put(378,57){$\textcolor{blue}5$}
\put(377,32){$\textcolor{red}2$}
\end{picture}

\includegraphics[width=13cm]{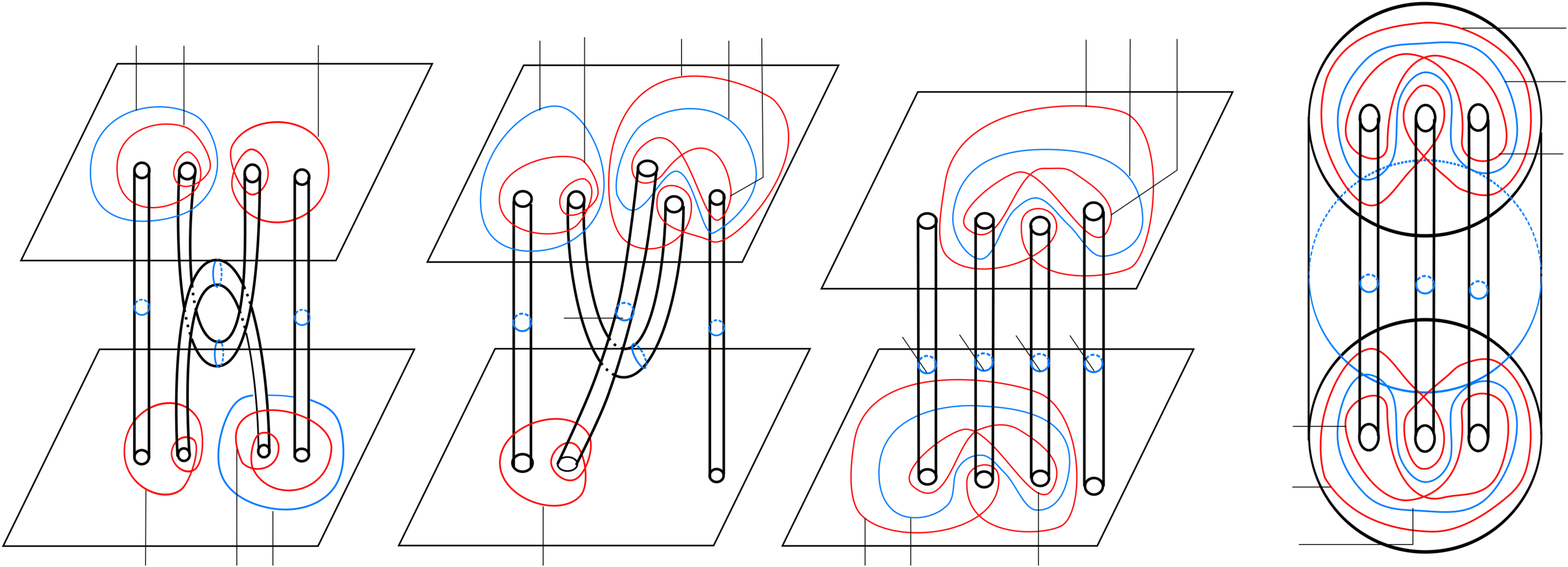}
\begin{picture}(400,0)(0,0)

\put(105,67){$\rightarrow$}
\put(195,67){$\rightarrow$}
\put(286,67){$\rightarrow$}

\put(50,130){$\textcolor{blue}6$}
\put(60,130){$\textcolor{red}3$}
\put(90,130){$\textcolor{red}1$}
\put(45,65){$\textcolor{blue}4$}
\put(67,83){$\textcolor{blue}2$}
\put(67,47){$\textcolor{blue}3$}
\put(90,65){$\textcolor{blue}1$}
\put(50,5){$\textcolor{red}2$}
\put(70,5){$\textcolor{red}4$}
\put(80,5){$\textcolor{blue}5$}

\put(138,131){$\textcolor{blue}6$}
\put(149,131){$\textcolor{red}3$}
\put(170,131){$\textcolor{red}1$}
\put(180,131){$\textcolor{blue}1$}
\put(189,131){$\textcolor{red}4$}
\put(128,64){$\textcolor{blue}4$}
\put(141,64){$\textcolor{blue}2$}
\put(165,47){$\textcolor{blue}3$}
\put(184,65){$\textcolor{blue}5$}
\put(140,5){$\textcolor{red}2$}

\put(260,131){$\textcolor{red}1$}
\put(270,131){$\textcolor{blue}1$}
\put(280,131){$\textcolor{red}4$}
\put(217,63){$\textcolor{blue}6$}
\put(230,63){$\textcolor{blue}2$}
\put(243,63){$\textcolor{blue}3$}
\put(255,63){$\textcolor{blue}5$}
\put(210,5){$\textcolor{red}2$}
\put(220,5){$\textcolor{blue}4$}
\put(250,5){$\textcolor{red}3$}

\put(370,130){$\textcolor{red}1$}
\put(370,115){$\textcolor{blue}1$}
\put(370,100){$\textcolor{red}4$}
\put(305,72){$\textcolor{blue}6$}
\put(316,72){$\textcolor{blue}2$}
\put(329,72){$\textcolor{blue}3$}
\put(354,72){$\textcolor{blue}5$}
\put(300,40){$\textcolor{red}2$}
\put(300,25){$\textcolor{red}3$}
\put(300,10){$\textcolor{blue}4$}

\end{picture}

\caption{The fibers of $q_f$ over the points of the boundary of $X$ (top left) agree with 
those of $q_{f'}$ over the points of the boundary of $\Nbd (v;Q_{f'})$ (top right) 
after a handle slide (bottom).}
\label{figure:handleslide}
\end{figure}
This implies that  $q_{f}|_{\partial V}$ coincides with $q_{f'}|_{\partial V}$.  
Consequently, $f$ coincides with $f'$ on the boundary of $V$. 
Thus, by gluing the maps 
$f|_{M- \Int V}$ and $f' |_V$ and then 
smoothening it appropriately, 
we obtain a stable map $g: (M,L) \to \Real^2$ satisfying 
$|\mathrm{II}^{2}(g)|=|\mathrm{II}^{2}(f)| - 2$ and $|\mathrm{II}^{3}(g)|=1$. 
Applying the same argument for the other local parts of $Q_f$  homeomorphic to $X$ repeatedly, 
we finally get a desired stable map of $(M, L)$ into $\Real^2$. 
\end{proof}

\section{The main theorem}
\label{sec:The main theorem}


We first show the "only if'' part of Theorem~\ref{thm:main theorem}. 
Suppose that there exists a stable map $f:(S^3,L) \to \Real^2$ with 
$\mathrm{II}^{2} = \emptyset$ and  $|\mathrm{II}^{3}| = 1$. 
Let $P$ be a branched polyhedron obtained by replacing a regular  
neighborhood of each vertex of type $\mathrm{II}^3$ of $Q_f$ with 
$X$. 
By Theorem~\ref{thm:from a Stein factorization to a branched shadow}, 
$P$ is a branched shadow of $(S^3, L)$. 
We note that, by this construction, the two vertices of $P$ is contained in 
a single component of the singular set $S(P)$. 
Let $c$ be the component of $S(P)$ containing the vertices of $P$. 
Set $P'=\Nbd(c;P) \cup X$, which is also a branched polyhedron, where 
every region is a half-open annulus. 
By Ishikawa-Koda \cite[Lemma~5.5]{IK17}, 
there exists a branched shadow $P''$ of $(S^3, L)$ 
that is obtained from $P'$ by attaching {\it towers} to some of its boundary components.
Here, a tower is a simple polyhedron obtained by gluing several $2$-disks to an annulus as depicted in 
Figure \ref{figure:tower}. 
\begin{figure}[htbp]
\centering\includegraphics[width=2cm]{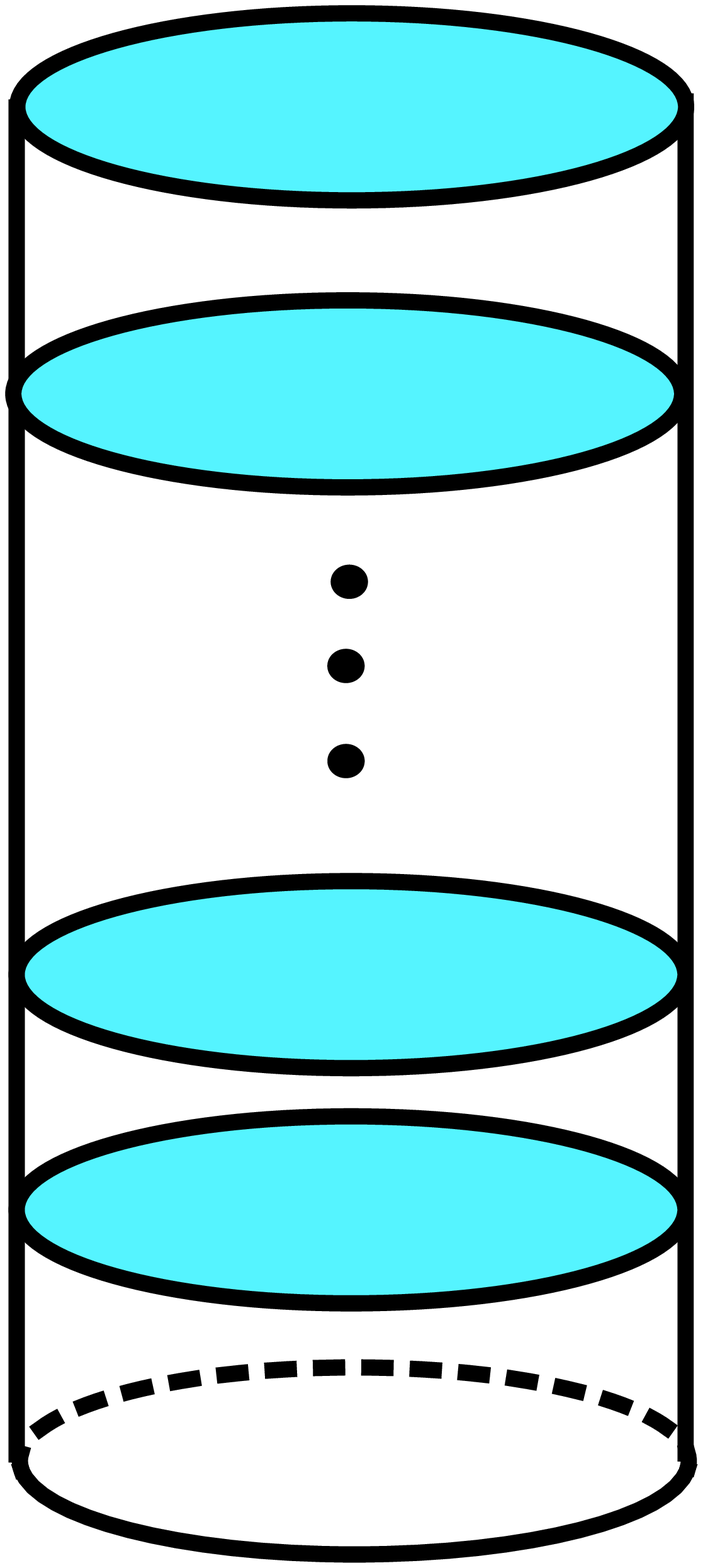}
\caption{A tower.}
\label{figure:tower}
\end{figure}
The branching of $P'$ naturally extends to that of $P''$. 

\begin{claim}
\label{claim:classification of P prime}
$P'$ is homeomporphic to one of the five branched polyhedra $P_1, P_2, \ldots, P_5$ depicted in 
Figure $\ref{figure:list_SP}$. 
\end{claim}
\begin{figure}[htbp]
\centering\includegraphics[width=13cm]{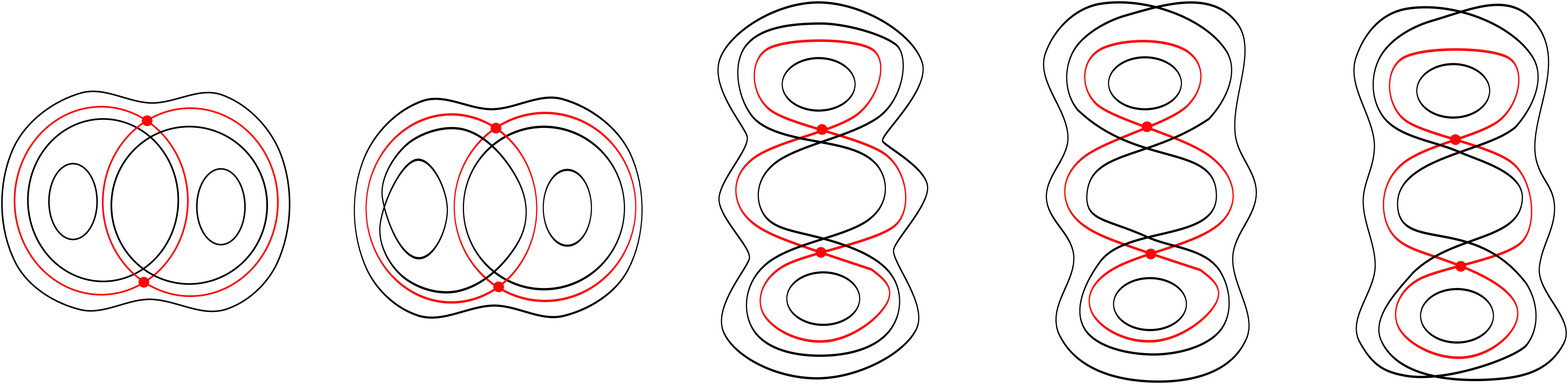}
\begin{picture}(400,0)(0,0)
\put(44,53){$+1$}
\put(42,0){$P_{1}$}
\put(126,52){$+1$}
\put(129,0){$P_{2}$}
\put(203,56){$+1$}
\put(204,0){$P_{3}$}
\put(282,55){$+1$}
\put(282,0){$P_{4}$}
\put(351,54){$+1$}
\put(354,0){$P_{5}$}

\end{picture}
\setlength\abovecaptionskip{10pt}
\caption{The branched polyhedra $P_1, P_2, \ldots, P_5$.}
\label{figure:list_SP}
\end{figure}
\begin{proof}[Proof of Claim~$\ref{claim:classification of P prime}$]
Since $c$ is a $4$-valent graph with exactly $2$ vertices, 
$c$ is homeomorphic to one of the graphs of types $1$, $2$ and $3$ shown in 
Figure \ref{figure:list_c}. 
\begin{figure}[htbp]
\centering\includegraphics[width=6.2cm]{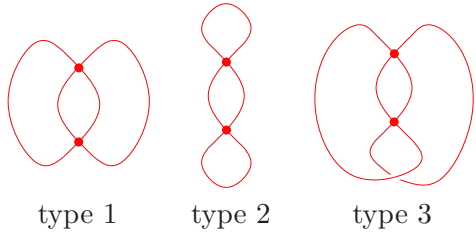}
\begin{picture}(400,0)(0,0)
\put(125,0){type $1$}
\put(183,0){type $2$}
\put(244,0){type $3$}
\end{picture}
\setlength\abovecaptionskip{10pt}
\caption{The $4$-valent graphs of types $1$, $2$ and $3$.}
\label{figure:list_c}
\end{figure}
Suppose that $c$ is of type $3$. 
Using the combinatorial structure of $c$, 
$P'$ is then obtained from $X$ by gluing subspaces of the boundary of $X$ 
according to the combinatorial structure around the edges of $P'$, 
see Figure \ref{figure:douitusi_c} for example. 
\begin{figure}[htbp]
\centering\includegraphics[width=4.5cm]{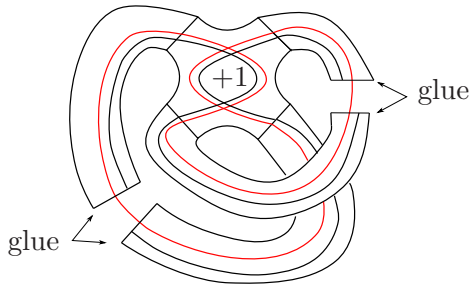}
\begin{picture}(400,0)(0,0)
\put(268,83){glue}
\put(113,26){glue}
\put(190,88){$+1$}
\end{picture}
\caption{From $X$ to $P'$.}
\label{figure:douitusi_c}
\end{figure}
However, as indicated in Figure \ref{figure:missmatch}, the branching of $X$ cannot be extended to 
that of $P'$ no matter how we glue subspaces of the boundary of $X$ in the above process. 
This contradicts the fact that the branched polyhedron $P'$ contains $X$. 
\begin{figure}[htbp]
\centering\includegraphics[width=9cm]{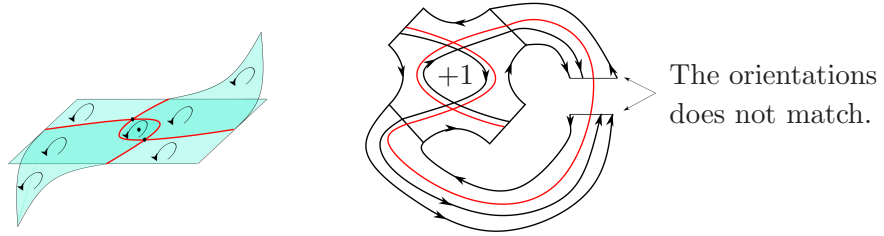}
\begin{picture}(400,0)(0,0)
\put(242,70){$+1$}
\put(330,70){The orientations}
\put(330,55){does not match.}
\end{picture}
\caption{If $c$ is of type $3$, the branching of $X$ cannot be extended to that of $P'$.}
\label{figure:missmatch}
\end{figure}
Therefore, $c$ must be of type $1$ or $2$.  
Then using the branching of $X$, we can check that $P'$ is one of the polyhedra depicted in 
Figures \ref{figure:Nbd_c_type1} and \ref{figure:Nbd_c_type2}. 
\begin{figure}[htbp]
\centering\includegraphics[width=13cm]{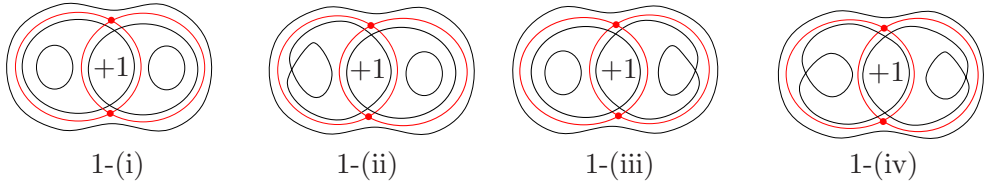}
\begin{picture}(400,0)(0,0)
\put(48,37){$+1$}
\put(145,36){$+1$}
\put(240,36){$+1$}
\put(341,35){$+1$}
\put(47,0){1-(i)}
\put(140,0){1-(ii)}
\put(234,0){1-(iii)}
\put(334,0){1-(iv)}
\end{picture}
\caption{The candidates of $P'$ when $c$ is of type $1$.}
\label{figure:Nbd_c_type1}
\end{figure}
\begin{figure}[htbp]
\centering\includegraphics[width=13cm]{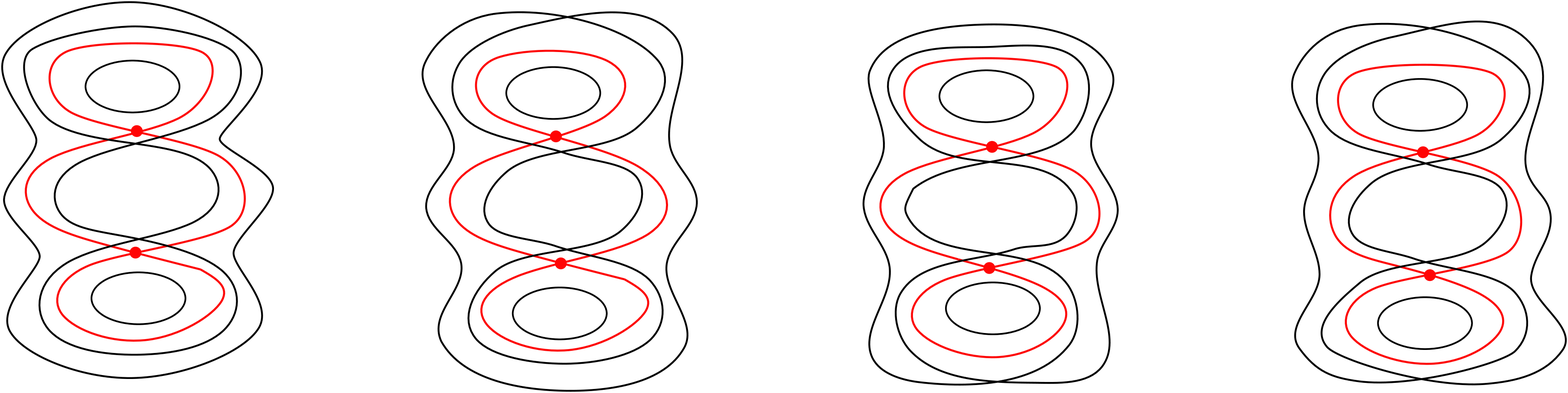}
\begin{picture}(400,0)(0,0)
\put(40,57){$+1$}
\put(140,55){$+1$}
\put(241,53){$+1$}
\put(343,52){$+1$}
\put(38,0){2-(i)}
\put(136,0){2-(ii)}
\put(236,0){2-(iii)}
\put(339,0){2-(iv)}
\end{picture}
\caption{The candidates of $P'$ when $c$ is of type $2$.}
\label{figure:Nbd_c_type2}
\end{figure}
We note that after replacing the color $i$ (internal) with $f$ (false) of the boundary-decoration 
of $P''$, $P''$ becomes a shadow of $S^3$. Thus, $P''$ is simply connected by 
Ishikawa-Koda \cite[Lemma~5.1]{IK17}. 
Therefore, $P'$ must be able to be simply connected by 
attaching towers to some of its boundary components. 
However, the branched polyhedron 1-(iv) in Figure \ref{figure:Nbd_c_type1} 
cannot be simply connected even though 
we attach a tower to every boundary component.
Hence, we can remove 1-(iv) from the 
possible list of shapes of $P'$. 
Further, the branched polyhedra 1-(ii) and 2-(ii) are homeomorphic to 
1-(iii) and 2-(iii), respectively. 
Consequently, $P'$ is homeomorphic to one of 1-(i), 1-(ii), 2-(i), 2-(ii) and 2-(iv). 
These are exactly the branched polyhedra $P_1, P_2, \ldots, P_5$ depicted in Figure \ref{figure:list_SP}. 
\end{proof}

\begin{claim}
\label{claim:P1, P2,...,P5 re branched shadows of L1, L2,...,L5}
The branched polyhedra $P_1, P_2, \ldots, P_5$ depicted in Figure $\ref{figure:list_SP}$ 
are, respectively, branched shadows of the exterior of the links 
$L_1'$, $L_2', \ldots, L_5'$,  where 
$L_1', L_2', L_3', L_4'$ are the links depicted in Figure $\ref{figure:Furutanilink}$, and 
$L'_5$ is the one depicted in Figure $\ref{figure:L_prime_5}$. 
\begin{figure}[htbp]
\centering\includegraphics[width=3.5cm]{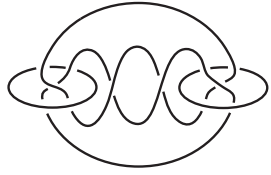}
\caption{The link $L'_5$.}
\label{figure:L_prime_5}
\end{figure}
\end{claim}
\begin{proof}[Proof of Claim~$\ref{claim:P1, P2,...,P5 re branched shadows of L1, L2,...,L5}$]
Let $U_1, U_2 , \ldots, U_5$ be diagrams of the trivial knot or the Hopf link 
illustrated in Figure \ref{figure:trivial2chain}. 
\begin{figure}[htbp]
\centering\includegraphics[width=13cm]{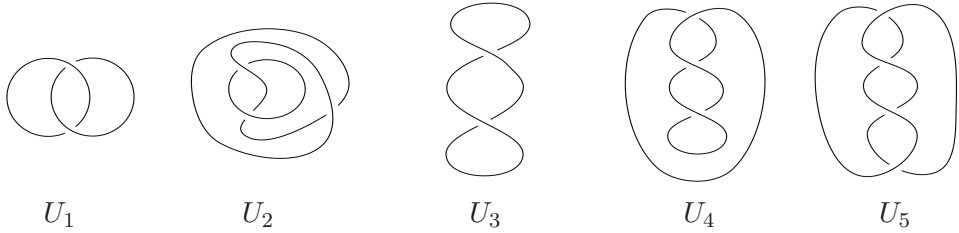}
\begin{picture}(400,0)(0,0)
\put(35,0){$U_{1}$}
\put(110,0){$U_{2}$}
\put(196,0){$U_{3}$}
\put(277,0){$U_{4}$}
\put(351,0){$U_{5}$}
\end{picture}
\caption{The diagrams $U_1, U_2, \ldots, U_5$ of the trivial knot or the Hopf link.}
\label{figure:trivial2chain}
\end{figure}
Let $P_{U_1}^*$, $P_{U_2}$, $P_{U_3}^*$, $P_{U_4}$, $P_{U_5}$ be the branched polyhedra obtained from 
$U_1, U_2 , \ldots, U_5$ as in Example \ref{ex:branched shadows of links}. 

We first show that $P_1$ is a branched shadow of $E(L'_1)$. 
Define a coloring of $\partial P^*_{U_1}$ by $\partial_e P^*_{U_{1}}=\partial P^*_{U_1}$. 
The boundary $\partial P^{\ast}_{U_{1}}$ of $P^{\ast}_{U_{1}}$ forms 
a  link in $S^{3}$, and  
$P^{\ast}_{U_{1}}$ is a branched shadow of the exterior $E(\partial P^{\ast}_{U_{1}})$ of 
the link $\partial P^{\ast}_{U_{1}} \subset S^{3}$. 
The branched polyhedron $P_{1}$ is obtained by removing the interior of the $2$-disks $D_1$ and $D_2$ 
shown in Figure \ref{figure:shadow_L1}. 
Recall the map $\pi_{P^*_{U_{1}}}$ defined in Section \ref{subsec:Branched shadows}. 
The preimages of the disks $D_1$ and $D_2$ under $\pi_{P^*_{U_{1}}}$ are then solid tori 
shown in Figure \ref{figure:shadow_L1}. 
\begin{figure}[htbp]
\centering\includegraphics[width=4cm]{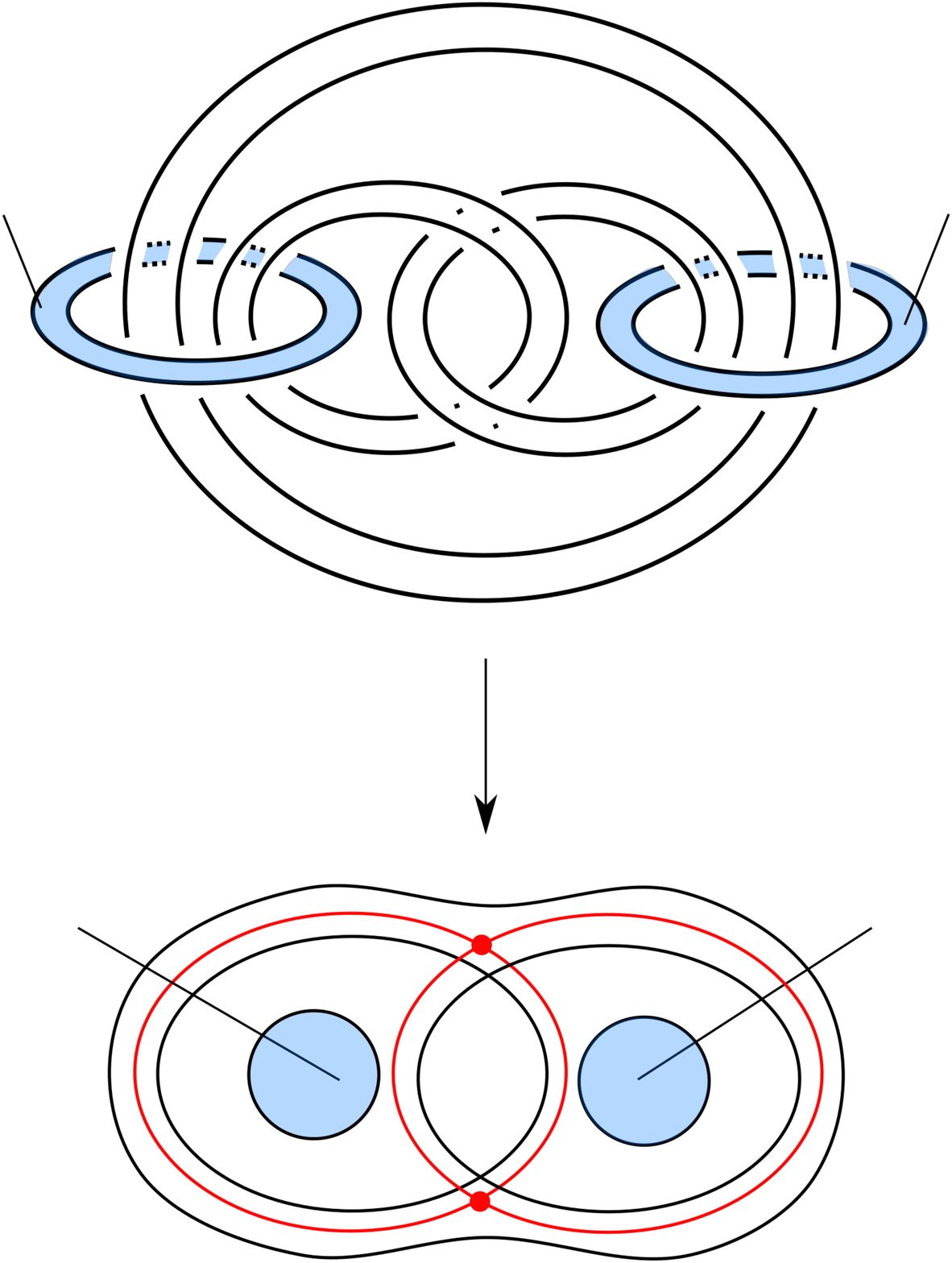}
\begin{picture}(400,0)(0,0)
\put(205,76){$\pi_{P^*_{U_{1}}}$}
\put(133,55){$D_{1}$}
\put(100,140){$\pi_{P^*_{U_{1}}}^{-1}(D_{1})$}
\put(250,55){$D_{2}$}
\put(260,140){$\pi_{P^*_{U_{1}}}^{-1}(D_{2})$}
\put(193,36){$+1$}
\end{picture}
\caption{The $2$-disks $D_{1}, D_{2}$ and their preimages under $\pi_{P^*_{U_{1}}}$.}
\label{figure:shadow_L1}
\end{figure}
As indicated in Figure \ref{figure:shadow_L1}, 
the space obtained from $E(\partial P^{\ast}_{U_{1}})$ 
by removing the interiors of the solid tori 
$\pi^{-1}(D_{1})$ and $\pi^{-1}(D_{2})$
is homeomorphic to $E(\partial P^*_{U_{1}})$. 
This implies that $P_{1}$ is a branched shadow of $E(L'_{1})$. 

The proof for the other cases runs in the same way. 
Indeed, the polyhedra $P_2$, $P_3$, $P_4$, $P_5$ can be obtained by removing the interiors of 
the $2$-disks from $P_{U_{2}}$, $P^*_{U_{3}}$, $P_{U_{4}}$, $P_{U_{5}}$ shown in Figure \ref{figure:P_i}. 
Since the the preimages of those disks are as drawn in the same figure, 
we see that $P_2$, $P_3$, $P_4$, $P_5$ are branched shadows of $E(L'_{2})$, $E(L'_{3})$, 
$E(L'_{4})$, $E(L'_{5})$, respectively. 
\begin{figure}[htbp]
\centering\includegraphics[width=13cm]{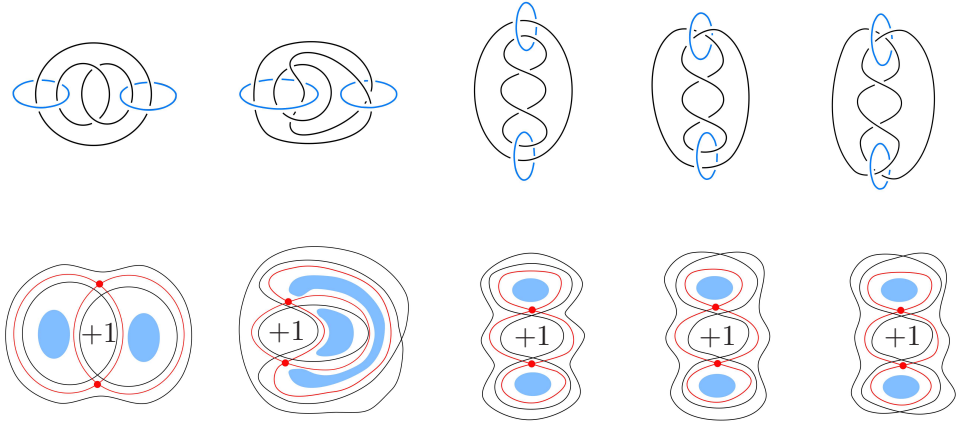}
\caption{The preimages of $2$-disks.}
\begin{picture}(400,0)(0,0)
\put(50,70){$+1$}
\put(121,71){$+1$}
\put(214,69){$+1$}
\put(284,69){$+1$}
\put(354,69){$+1$}
\end{picture}
\label{figure:P_i}
\end{figure}
\end{proof}
By Claim~\ref{claim:classification of P prime}, $P''$ is homeomorphic to 
one of $P_1, P_2, \ldots, P_5$ depicted in Figure \ref{figure:list_SP}, or 
the one obtained by attaching towers to some of their boundary components. 
Thus, by Claim~\ref{claim:P1, P2,...,P5 re branched shadows of L1, L2,...,L5}, 
together with Costantino-Thurston \cite[Proposition 3.27]{CT08}, 
$E(L)$ is diffeomorphic to
a $3$-manifold obtained by Dehn filling $E(L_1'), E(L_2'), \ldots, E(L_5')$  along some of 
$($possibly none of$)$ boundary tori. 
Since SnapPy \cite{CDGW} tells us that the complements of the 
links $L'_3$ and $L'_5$ are homeomorphic, we can exclude $E(L_5')$ from them.

Next, we show the ``if'' part of Theorem~\ref{thm:main theorem}. 
Suppose that $E(L)$ is diffeomorphic to
a $3$-manifold obtained by Dehn filling $E(L_1'), E(L_2'), E(L_3'), E(L_4')$  along some of 
$($possibly none of$)$ boundary tori. 
Let $P_i$ ($i = 1,2,3,4$) be the branched shadow of $E(L'_i)$ shown in Figure \ref{figure:list_SP}. 
Then again by Costantino-Thurston \cite[Proposition 3.27]{CT08}, we can obtain a branched shadow $P$ 
of $E(L)$ by attaching towers to some of (possibly, none of) the boundary components of $P_i$. 
By Theorem~\ref{thm:bsc is equal to smc} to $P$, we can construct a stable map 
$f: (S^3, L) \to \Real^2$ from $P$. 
Then by Corollary~\ref{cor:replacement stable maps}, 
there exists a stable map $g: (S^3, L) \to \Real^2$ such that 
the Stein factorization $Q_g$ is obtained from $Q_f$ by replacing 
each part homeomorphic to $X$ with 
the model of Figure $\ref{figure:steindecomodel}$ (iv). 
Clearly, we have $\mathrm{II}^{2}(g) = \emptyset$ and  $|\mathrm{II}^{3}(g)| = 1$. 
This completes the proof of Theorem~\ref{thm:main theorem}. 

\vspace{1em}

As we have seen in the above proof, 
for each link $L_i'$ ($i = 1, 2, 3, 4$), we can construct a stable map 
$f_i: (S^3, L_i') \to \Real^2$ with $\mathrm{II}^2 (f_i) = \emptyset$ and $|\mathrm{II}^3 (f_i)| = 1$ 
explicitly. 
In particular, in this case, as we have seen in 
Example \ref{ex:stable maps of links}, 
$Q_{f_i}$ is obtained from $P_i$ by replacing the unique local part homeomorphic to 
$X$ with the model of Figure $\ref{figure:steindecomodel}$ {\rm (iv)}, and further, we can draw  
the configuration of fibers of $f_i$ explicitly. 
Consequently, we have the following. 

\begin{corollary}
\label{cor:configuration of the unique singular fiber of type II3}
For each link $L_i'$ $($$i = 1, 2, 3, 4$$)$ shown in Figure $\ref{figure:Furutanilink}$, 
there exists a stable map $f_i: (S^3, L'_i) \to \Real^2$ with 
$\mathrm{II}^2 (f_i) = \emptyset$ and $|\mathrm{II}^3 (f_i)| = 1$
such that 
the configuration of the unique singular fiber of type $\mathrm{II}^3$ 
is as shown in Figure $\ref{figure:II3}$. 
\end{corollary}
\begin{figure}[htbp]
\centering\includegraphics[width=14cm]{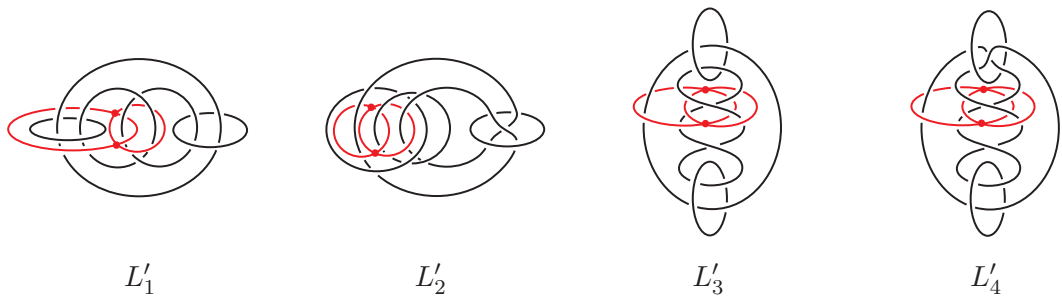}
\begin{picture}(400,0)(0,0)
\put(45,0){$L'_{1}$}
\put(155,0){$L'_{2}$}
\put(260,0){$L'_{3}$}
\put(365,0){$L'_{4}$}
\end{picture}
\caption{The configuration of the type $\mathrm{II}^3$ singular fibers.}
\label{figure:II3}
\end{figure}

\section*{Acknowledgments} 

\thanks{
The authors would like to thank the anonymous referee
for his or her valuable comments and suggestions 
that helped them to improve the exposition. 
Y. K. is supported by JSPS KAKENHI Grant Numbers JP17H06463, 
JP20K03588 and JST CREST Grant Number JPMJCR17J4.}

\appendix 
\section{Hyperbolic structures}
\label{appendix:Hyperbolic structures}

Let $P$ be a (branched) shadow of a $3$-manifold $M$, and let $c$ be a component of $S(P)$. 
Set $P' := \Nbd (c; P)$. 
It is easily seen that if $c$ contains no vertices of $P$, $\pi_P^{-1} (P')$ is a 
pair of pants bundle over $S^1$. 
In particular, $\pi_P^{-1} (P')$ is a Seifert fibered space in this case. 
Suppose that $c$ contains $n$ vertices of $P$, where $n > 0$. 
In this case, it was shown by Costantino-Thurston \cite[Proposition 3.33]{CT08} that $\Int \pi_P^{-1} (P')$ 
is a hyperbolic manifold of volume $2n v_{\mathrm{oct}}$. 
In fact, for each vertex in $c$, we can construct two ideal regular octahedra glued together along some 
of their faces to form a handlebody of genus three, and the remaining faces are glued together using the combinatorial structure along the edges of $P'$. 
See also Ishikawa-Koda \cite[Section 6]{IK17}. 

Let $X(P)$ be the set of local parts of $P$ homeomorphic to $X$  depicted in Figure \ref{figure:X} 
whose vertices are contained in $c$. 
Recall that a branched shadow obtained as in 
Theorem~\ref{thm:from a Stein factorization to a branched shadow} 
always contains this type of local part, each of which corresponds to 
a singular fiber of type $\mathrm{II}^3$. 
Suppose that $X(P) \neq \emptyset$, and $X(P) = \{ X_1, X_2, \ldots, X_m\}$. 
Set $P'' := P' \cup \bigcup_i X_i$. 
Then it was mentioned in Costantino-Thurston \cite[Theorem~3.38]{CT08} that 
$\Int \pi_P^{-1} (P'')$ 
admits a hyperbolic structure of volume $2 (n-2m) v_{\mathrm{oct}} + 10 m v_{\mathrm{ted}}$. 
In fact, we can show that for each vertex in $P'' - P'$, we can construct two ideal regular octahedra 
glued together as above, 
and for each $X_i$, we can construct ten ideal regular tetrahedra glued together along some 
of their faces to form a handlebody of genus three. 
Then the remaining faces, all of them are totally geodesic ideal regular triangles, 
are glued together using the combinatorial structure along the edges of $P''$. 
We believe it is worth giving an explicit description of this fact here. 
The following proposition is essential and enough to see that.  
 
 \begin{proposition}
 \label{prop:decomposition into ten ideal regular tetrahedra}
 Let $P_1$ be the shadowed polyhedron illustrated in Figure $\ref{figure:list_SP}$ equipped with 
 the coloring of $\partial P_1$ by $\partial_e(P_1) = \partial P_1$.  
 Let $M$ be the compact $3$-manifold represented by $P_1$.  
 Then $\Int M$ is a hyperbolic $3$-manifold that admits a decomposition 
 into ten ideal regular tetrahedra. 
Further,  let $Y_1$ and $Y_2$ be the cones over three points in $P_1$ that cut $P_1$ into 
$X$. 
Then $\pi_{P_1}^{-1} (Y_i) \cap \Int M$ $(i=1, 2)$ is a totally geodesic open pair of pants 
consisting of two faces of those ideal tetrahedra. 
\end{proposition}
\begin{proof}
As we have seen in the proof of Theorem~\ref{thm:main theorem}, 
$M = E(L_1')$. 
By Dunfield-Thurston \cite{DT03}, the complement of $L'_{1}$ 
can be decomposed into ten ideal regular tetrahedra as follows. 

Let $\tau :E(L'_{1}) \to E(L'_{1})$ be the involution as shown in Figure \ref{figure:tau}. 
\begin{figure}[htbp]
\centering\includegraphics[width=6cm]{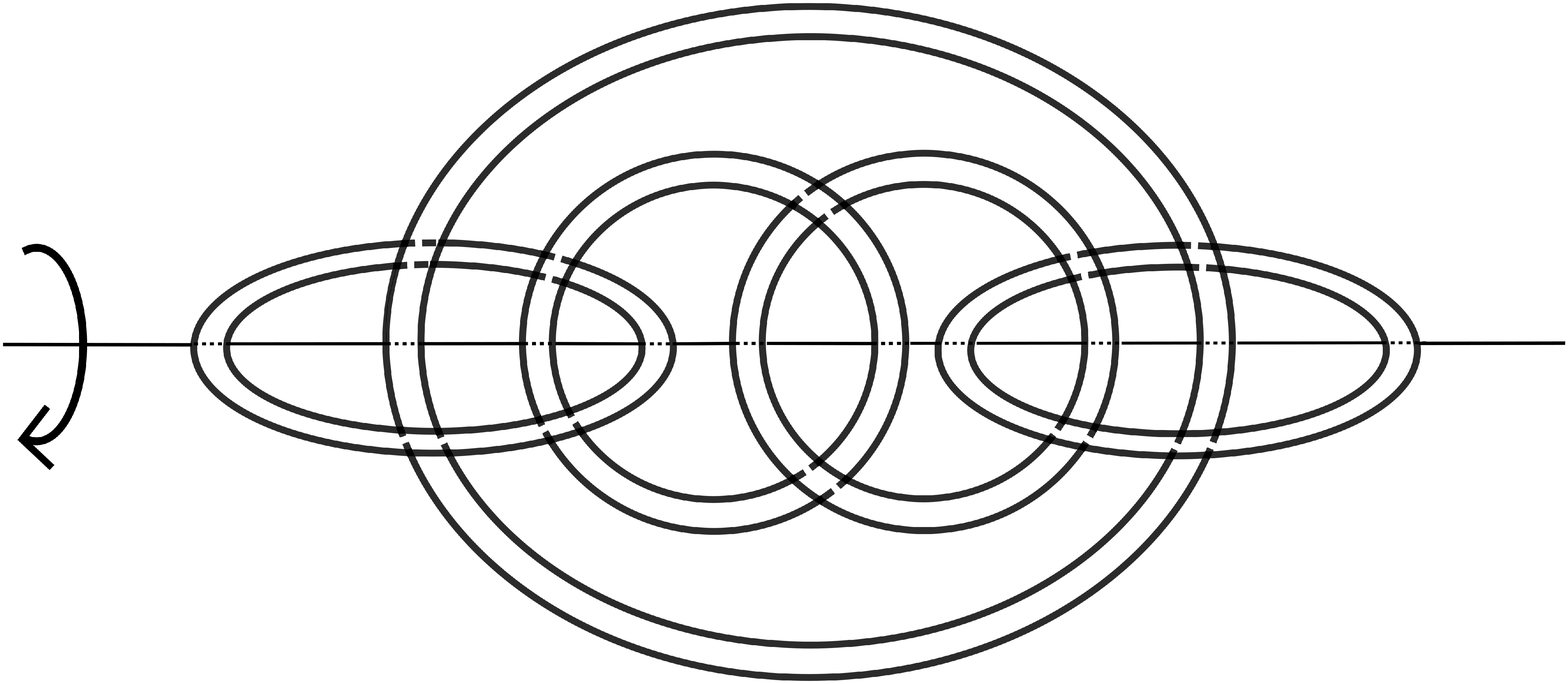}
\begin{picture}(400,0)(0,0)
\put(116,35){$\tau$}
\end{picture}
\caption{The involution $\tau$.}
\label{figure:tau}
\end{figure}
Consider the double branched cover $E(L'_{1}) \to E(L'_{1})/\tau$. 
We think of $S^3$ as the boundary of the $4$-simplex $\Delta^4$. 
Then as shown in Figure \ref{figure:5chainlink_eq}, 
$E(L'_{1})/ \tau$ is homeomorphic to the boundary of $\Delta^4$ minus 
an open ball neighborhood 
of each vertex of $\Delta^4$, and 
the image of the set of fixed points of $\tau$ corresponds to the set of edges of $\Delta^4$. 
\begin{figure}[htbp]
\centering\includegraphics[width=10cm]{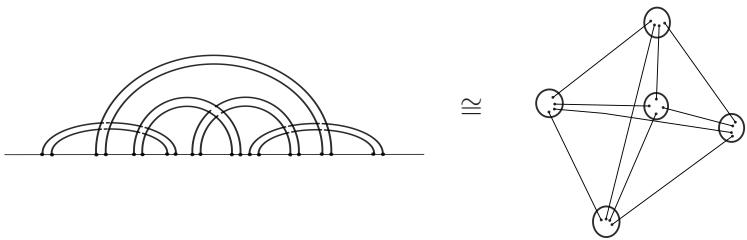}
\begin{picture}(400,0)(0,0)
\put(230,60){$\cong$}
\end{picture}
\caption{The quotient space $E(L'_{1})/\tau$.}
\label{figure:5chainlink_eq}
\end{figure}
In this way, we obtain a decomposition of $E(L'_{1})/ \tau$ into five truncated tetrahedra such that 
each edge is touched by three faces. 
This decomposition lifts to a decomposition of $E(L'_{1})$ into truncated tetrahedra 
whose dihedral angles are all $\pi / 3$. 
Thus, this induces a decomposition of $\Int M$ into ten ideal regular hyperbolic tetrahedra. 

The preimages $\pi_{P_1}^{-1} (Y_1)$ and $\pi_{P_1}^{-1} (Y_2)$ 
of $Y_1$ and $Y_2$ are pair of pants as shown in 
Figure \ref{figure:XsubsetP1}. 
\begin{figure}[htbp]
\centering\includegraphics[width=12cm]{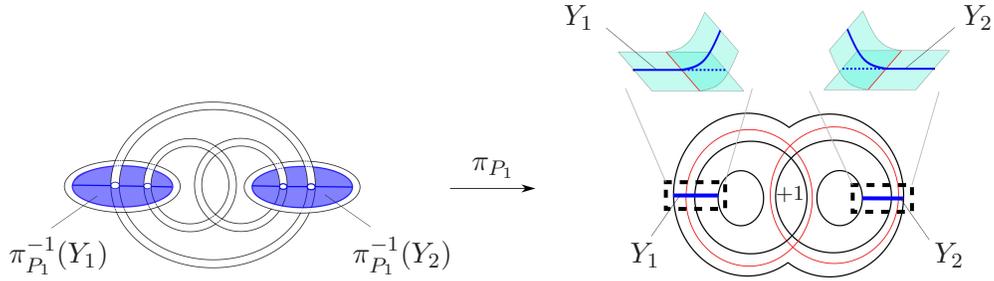}
\caption{The preimages $\pi_{P_1}^{-1} (Y_1)$ and $\pi_{P_1}^{-1} (Y_2)$.}
\begin{picture}(400,0)(0,0)
\put(15,40){$\pi_{P_1}^{-1} (Y_1)$}
\put(145,40){$\pi_{P_1}^{-1} (Y_2)$}
\put(190,75){$\pi_{P_1}$}
\put(305,64){\scriptsize $+1$}
\put(250,40){$Y_1$}
\put(365,40){$Y_2$}
\put(225,130){$Y_1$}
\put(376,130){$Y_2$}
\end{picture}
\label{figure:XsubsetP1}
\end{figure}
Each of those pair of pants consists of exactly two faces of the truncated tetrahedra 
constructed above. 
Thus, $\pi_{P_1}^{-1} (Y_i) \cap \Int M$ $(i=1, 2)$ is a totally geodesic open pair of pants 
consisting of two faces of the above ideal regular tetrahedra. 
\end{proof}
  
As we have seen in Claim~\ref{claim:classification of P prime} of the proof of 
Theorem~\ref{thm:main theorem}, the polyhedra $P_2, P_3, P_4$ 
illustrated in Figure $\ref{figure:list_SP}$ are all obtained by cutting 
$P_1$ along $Y_1 \cup Y_2$ and then re-gluing the cut ends appropriately. 
Thus, we have the following.

 \begin{corollary}
 \label{cor:volume}
 The links $L_1', L_2', L_3', L_4' \subset S^3$ in Figure  $\ref{figure:Furutanilink}$ are all hyperbolic 
 of volume $10v_{\mathrm{tet}}$. 
\end{corollary}

By the proof of  Proposition \ref{prop:decomposition into ten ideal regular tetrahedra}, 
we also have the following. 

\begin{corollary}
\label{cor:cusp shapes}
The cusp shapes of the complements of the 
hyperbolic links $L_1', L_2', L_3', L_4'  \subset S^3$ in Figure $\ref{figure:Furutanilink}$  
are as shown in Figure $\ref{figure:cuspshape}$. 
\end{corollary}

\begin{figure}[htbp]
\centering\includegraphics[width=14cm]{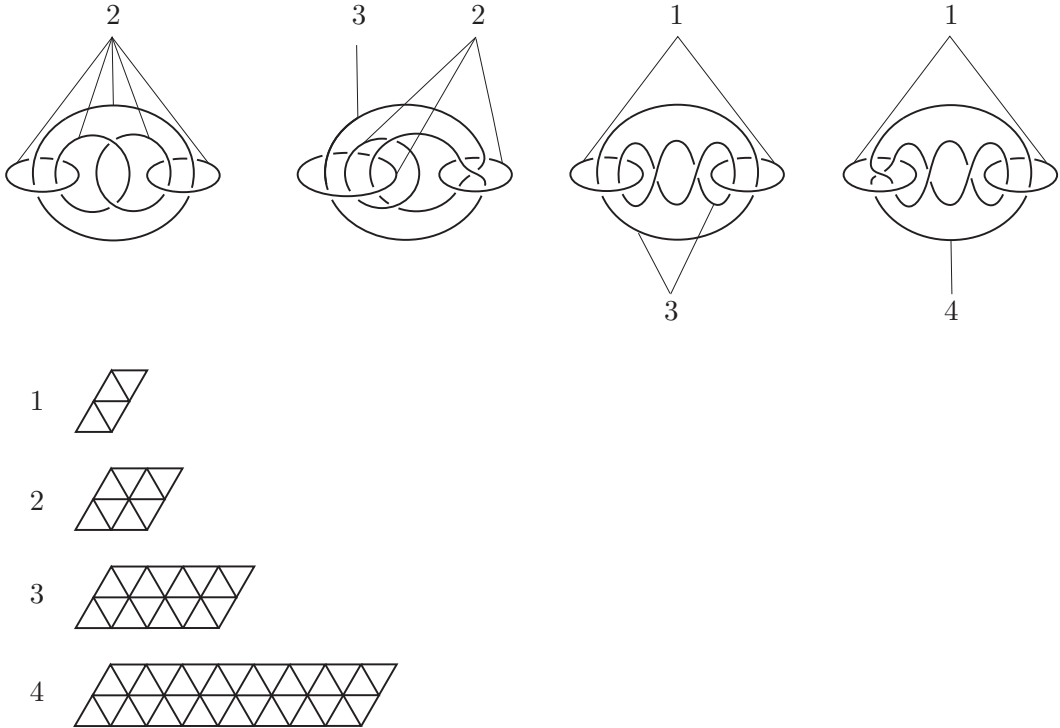}
\begin{picture}(400,0)(0,0)
\put(10,133){$1$}
\put(10,95){$2$}
\put(10,60){$3$}
\put(10,23){$4$}
\put(39,279){$2$}
\put(132,279){$3$}
\put(177,279){$2$}
\put(252,279){$1$}
\put(250,167){$3$}
\put(356,279){$1$}
\put(356,167){$4$}
\end{picture}
\caption{The cusp shapes of the complements of $L_1', L_2', L_3', L_4' $.}
\label{figure:cuspshape}
\end{figure}


\begin{thebibliography}{99999}

\bibitem{BR74}
Burlet, O., de Rham, G., 
Sur certaines applications g\'{e}n\'{e}riques d'une vari\'{e}t\'{e} close \`{a} 3  dimensions dans le plan, 
Enseignement Math. (2)  {\bf 20} (1974), 275--292.

\bibitem {Cos05}
Costantino, F.,
{\it Shadows and branched shadows of $3$ and $4$-manifolds},
Edizioni della Normale, Scuola Normale Superiore, Pisa, Italy, 2005.


\bibitem {CT08}
Costantino, F., Thurston, D.,
3-manifolds efficiently bound 4-manifolds,
J. Topol. {\bf 1}(3), pp. 703--745 (2008).

\bibitem {CDGW}
Culler, M., Dunfield, N. M., Goerner, M.,  Weeks, J. R., 
SnapPy, a computer program for studying the geometry and topology of $3$-manifolds, 
available at \texttt{http://snappy.computop.org.}


\bibitem {DT03}
Dunfield, N., Thurston, W.P.,
The virtual Haken conjecture: experiments and examples,
Geom. Topol. {\bf 7}, pp. 399--441 (2003).

\bibitem {GG73}
Golubitsky, M., Guillemin, V.,
{\it Stable mappings and their singularities},
Graduate Texts in Mathematics, Vol. 14, Springer-Verlag, New York-Heidelberg, 1973. 

\bibitem{Gro09}
Gromov, M., 
Singularities, expanders and topology of maps. I. Homology versus volume in the spaces of cycles, 
Geom. Funct. Anal. {\bf 19} (2009), no. 2, 743--841.

\bibitem {IK17}
Ishikawa, M., Koda, Y.,
Stable maps and branched shadows of 3-manifold,
Math. Ann. {\bf 367}, pp. 1819--1863 (2017).

\bibitem{KS12}
Kalm\'{a}r, B., Stipsicz, A. I., 
Maps on 3-manifolds given by surgery, 
Pacific J. Math. {\bf 257} (2012), no. 1, 9--35. 	

\bibitem{KLP84}
Kushner, L., Levine, H., Porto, P., 
Mapping three-manifolds into the plane. I, 
Bol. Soc. Mat. Mexicana (2) {\bf 29} (1984), no. 1, 11--33. 

\bibitem {Lev65}
Levine, H.,
Elimination of cusps,
Topology {\bf 3}, suppl. 2, pp. 263--296 (1965).

\bibitem{Lev85}
Levine, H., 
{\it Classifying immersions into $\Real^4$ over stable maps of $3$-manifolds into $\Real^2$}, 
Lecture Notes in Mathematics {\bf 1157}, Springer-Verlag, Berlin, 1985. 

\bibitem {Mat71}
Mather, J.,
Stability of $C^{\infty}$ mappings, VI: The nice dimensions,
{\it Proceedings of Liverpool Singularities-Symposium}, I, Lecture Notes in Math., Vol. 192, Springer, Berlin, pp. 207--253 (1971).

\bibitem{Sae94} 
Saeki, O., 
Stable maps and links in 3-manifolds, 
Workshop on Geometry and Topology (Hanoi, 1993), Kodai Math. J. {\bf 17} (1994), no. 3, 518--529. 

\bibitem {Sae96}
Saeki, O.,
Simple stable maps of 3-manifolds into surfaces,
Topology {\bf 35}(3), pp. 671--698 (1996).

\bibitem {Sae04}
Saeki, O.,
{\it Topology of singular fibers of differentiable maps},
Lecture Notes in Mathematics, 1854, Springer-Verlag, Berlin, 2004.

\bibitem {Tur92}
Turaev, V.G.,
Shadow links and face models of statistical mechanics,
J. Differential Geom. {\bf 36}(1), pp. 35--74 (1992).

\bibitem {Tur94}
Turaev, V.G.,
{\it Quantum invariants of knots and $3$-manifolds},
de Gruyter Studies in Mathematics, vol. 18. Walter de Gruyter \& Co., Berlin (1994).
\end{thebibliography}
\end{document}